\documentclass[11pt]{amsart}
\usepackage{amsfonts,latexsym,amssymb,dsfont,amsmath,etex}
\usepackage{amsthm}
\usepackage{enumitem}
\usepackage[dvipsnames]{pstricks}
\usepackage{pst-solides3d}
\usepackage[all]{xy}
\usepackage{todonotes}
\usepackage{a4wide}
\usepackage{hyperref}

\hypersetup{
    colorlinks,
    linkcolor={red!50!black},
    citecolor={blue!50!black},
    urlcolor={blue!80!black}
}

\newcommand{\R}{{\mathbb R}}
\def\C{{\mathbb C}}
\renewcommand{\P}{{\mathbb P}}
\newcommand{\N}{{\mathbb N}}

\newcommand{\rmi}{\mathrm{i}}

\newcommand{\ep}{ \varepsilon  }

\newcommand{\vp}{\varphi  }

\newcommand{\D}{\mathcal{D}}

\newcommand{\ind}{\mathrm{ind}}
\newcommand{\supp}{\mathrm{supp}}
\newcommand{\APS}{\mathrm{APS}}
\newcommand{\loc}{\mathrm{loc}}

\newcommand{\id}{{\ensuremath{\mathds{1}}}}
\newcommand{\dV}{\mathrm{dV}}
\newcommand{\dA}{\mathrm{dA}}
\newcommand{\dt}{\mathrm{dt}}
\newcommand{\res}{\mathrm{res}}
\newcommand{\FE}{F\hspace{-0.8mm}E}
\newcommand{\gc}{\check g}
\newcommand{\Dc}{\check\D}
\newcommand{\Hc}{\check H}
\newcommand{\betac}{\check\beta}
\newcommand{\nablac}{\check\nabla}
\newcommand{\gd}{\hat g}
\newcommand{\Dd}{\hat D}
\newcommand{\Bd}{\hat B}
\newcommand{\betad}{\hat\beta}
\newcommand{\nablad}{\hat\nabla}
\newcommand{\Adach}{{\widehat{\mathrm{A}}}}
\newcommand{\TAdach}{{\mathrm{T}\widehat{\mathrm{A}}}}
\newcommand{\adach}{{\widehat{\mathrm{a}}}}
\newcommand{\ch}{\mathrm{ch}}
\newcommand{\SF}{\mathrm{sf}}
\newcommand{\Qpp}{Q_{++}(t_2,t_1)}
\newcommand{\Qpm}{Q_{+-}(t_2,t_1)}
\newcommand{\Qmp}{Q_{-+}(t_2,t_1)}
\newcommand{\Qmm}{Q_{--}(t_2,t_1)}
\newcommand{\aAPS}{\mathrm{aAPS}}
\newcommand{\spec}{\mathrm{spec}}
\newcommand{\T}{\mathcal{T}}
\newcommand{\G}{\mathcal{G}}

\newcommand{\ausblenden}[2]{#2}

\newcommand{\<}{\langle}
\renewcommand{\>}{\rangle}

\makeatletter
\newcommand{\leqnomode}{\tagsleft@true}
\newcommand{\reqnomode}{\tagsleft@false}
\makeatother

 \newtheorem{theorem}{Theorem}[section]
 \newtheorem{lem}[theorem]{Lemma}
 \newtheorem{cor}[theorem]{Corollary}

 \newtheorem*{maintheorem}{Main Theorem}

 \theoremstyle{definition}

 \newtheorem{rem}[theorem]{Remark}
 \newtheorem{rems}[theorem]{Remarks}

\allowdisplaybreaks
\title[An index theorem for Lorentzian manifolds]{An index theorem for Lorentzian manifolds with compact spacelike Cauchy boundary}

\author[C. B\"ar]{Christian B\"ar}
\address{Institute for Mathematics, Potsdam University, Karl-Liebknecht-Str.~24-25, 14476 Potsdam, Germany}
\email{baer@math.uni-potsdam.de}

\author[A. Strohmaier]{Alexander Strohmaier}
\address{School of Mathematics,  University of Leeds,  Leeds, LS2 9JT, UK} 
\email{a.strohmaier@leeds.ac.uk}

\date{\today}

\begin{document}

\begin{abstract}
 We show that the Dirac operator on a compact globally hyperbolic Lorentzian spacetime with spacelike Cauchy boundary is a Fredholm operator if appropriate boundary conditions are imposed.
 We  prove that the index of this operator is given by the same expression as in the index formula of Atiyah-Patodi-Singer for Riemannian manifolds with boundary.
 The index is also shown to equal that of a certain operator constructed from the evolution operator and a spectral projection on the boundary.
 In case the metric is of product type near the boundary a Feynman parametrix is constructed.
\end{abstract}
\maketitle


\section*{Introduction}

The Atiyah-Singer index theorem is one of main mathematical achievements of the 20$^\mathrm{th}$ century due to 
its many applications and because of the conceptual insights it provides.
It computes the index of an elliptic operator on a compact manifold without boundary.
An analog for hyperbolic operators on Lorentzian manifolds is unknown and is not to be expected.
Compact Lorentzian manifolds without boundary are for many reasons unsuitable in this context.
For example, they violate all causality conditions, which not only makes them unsuitable as models in General Relativity but also 
causes analytic problems for the operators; e.g.\ there is no well-posed initial value problem.

The Atiyah-Patodi-Singer index theorem deals with elliptic operators on compact manifolds with boundary.
The boundary conditions which one has to impose are based on the spectral decomposition for the induced operator on the boundary.
The aim of the present article is to derive a Lorentzian analog for this theorem.

We consider Lorentzian manifolds $M$ with boundary where the boundary consists of two smooth spacelike Cauchy hypersurfaces.
Since the boundary is Riemannian, Atiyah-Patodi-Singer boundary conditions still make sense.
Moreover, we assume a spin structure on the manifold so that spinors and the Dirac operator are defined.
In even dimensions the spinor bundle splits into the subbundles of spinors of positive and negative chirality, respectively, $SM=S^+M\oplus S^-M$.
The Dirac operator maps sections of positive-chirality spinors to those of negative chirality.
The domain $\FE^0_\APS(M;S^+M)$ of the Dirac operator can be characterized as the completion of smooth sections of $S^+M$ with respect to the $L^2$-graph norm for the Dirac operator, subject to Atiyah-Patodi-Singer boundary conditions.
Theorems~\ref{thm:IndQ=IndDAPS} and \ref{thm:IndGeom} combine to give the

\begin{maintheorem}
Let $(M,g)$ be a compact time-oriented globally hyperbolic Lorentzian manifold with boundary $\partial M = \Sigma_- \sqcup \Sigma_+$. 
Here $\Sigma_\pm$ are smooth spacelike Cauchy hypersurfaces, with $\Sigma_+$ lying in the future of $\Sigma_-$.
Assume that $M$ is even dimensional and comes equipped with a spin structure.

Then the Dirac operator $D_\APS: \FE^0_\APS(M;S^+M) \to  L^2(M;S^-M)$ under Atiyah-Patodi-Singer boundary conditions is Fredholm and its index is given by
$$
\ind[D_\APS]
= 
\int_M \Adach (\nabla) + \int_{\partial M}\TAdach(g) - \frac{h(A_{-})+h(A_{+})+\eta(A_{-})-\eta(A_{+})}{2}\, .
$$ 
\end{maintheorem}

The right hand side in the index formula is precisely the same as in the original Riemannian Atiyah-Patodi-Singer index theorem.
Here $\Adach(\nabla)$ is the $\Adach$-form manufactured from the curvature of the Levi-Civita connection of the Lorentzian manifold, $\TAdach(g)$ is the corresponding trangression form which also depends on the second fundamental form of the boundary (and vanishes if the boundary is totally geodesic).
Moreover, $A_\pm$ denotes the Dirac operator on $\Sigma_\pm$, $h$ the dimension of the kernel, and $\eta$ the $\eta$-invariant.

In the same way as for the Atiyah-Patodi-Singer index theorem the existence of a spin structure is not essential, but the statement of the theorem as well as its proof
easily generalize to the case of twisted  spin$^c$-Dirac operators. In order not to obscure the presentation by additional twisting bundles
we will state and discuss the more general theorem \ref{thm:maingeneral}  in the concluding remarks.

The Dirac operator on a Lorentzian manifold is far from hypoelliptic; 
solutions of the Dirac equation can have very low regularity.
Theorem~\ref{thm:KerDAPSglatt} tells us that under Atiyah-Patodi-Singer boundary conditions this is no longer so.
Solutions are always smooth as if we had elliptic regularity at our disposal.
Moreover,
$$
\ind[D_\APS] = \dim \ker[D_\APS|_{C^\infty(M;S^+M)}] -  \dim \ker[D_\aAPS|_{C^\infty(M;S^+M)}]\, .
$$ 
Here $D_\aAPS$ stands for the Dirac operator subject to ``anti-Atiyah-Patodi-Singer'' boundary conditions, the conditions complementary to the APS conditions.
The occurrence of the aAPS boundary conditions and the fact that $D_\aAPS$ again maps sections of $S^+M$ to those of $S^-M$, and not in the reverse direction, are different from the corresponding formula in the Riemannian setting.
In the elliptic case, the Dirac operator with anti-APS boundary conditions will in general have infinite-dimensional kernel.
In the Lorentzian situation, anti-APS conditions work equally well as the APS conditions.

Conceptually, one can think of this phenomenon as follows:
In the Riemannian case the Dirac equation subject to APS-conditions behaves like a heat equation being solved forward in time (which is well posed) while under aAPS-conditions it behaves like a heat equation being solved backward in time (which is ill posed).
In the Lorentzian setting the Dirac equation is essentially a wave equation which can be solved forward in time as well as backward in time.
On a more technical level, if the metric of the manifold has product structure near the boundary one can attach semi-infinite cylinders to replace the manifold with boundary by a complete one.
In the Riemannian setting solutions to the Dirac equation under APS-conditions correspond to exponentially decaying solutions of certain ODEs while aAPS-conditions lead to exponentially growing solutions which are not $L^2$ and hence bad.
In the Lorentzian case both boundary conditions lead to oscillating solutions which can be treated on equal footing.

The Atiyah-Singer index theorem and the Atiyah-Patodi-Singer index theorem have often been used in Quantum Field Theory, particularly in the context of the chiral anomaly. 
Here the underlying space is assumed to be Riemannian and the formulae are then applied to Lorentzian spacetimes by analogy. 
This is partially motivated by perturbative computations of gravitational corrections to the chiral anomaly that yield the same terms as in the Atiyah-Singer index formula. 
Our index theorem allows for a rigorous and non-perturbative derivation of a geometric formula for the chiral anomaly including an $\eta$-correction term, see \cite{baerstroh2015chiral}.

In contrast to traditional boundary conditions for second order operators such as Dirichlet or Neumann conditions which occur naturally in many physical situations, the APS conditions for first-order elliptic operators were introduced rather late and only for mathematical reasons.
They are conditions which make Dirac operators Fredholm and are therefore suitable for index theory.
In the Lorentzian setting however, APS conditions have a natural physical interpretation;
they can be read as allowing only Dirac-wave functions for particles or for antiparticles in the beginning and the end of a spacetime region respectively.
This is crucial for the application of our index theorem in \cite{baerstroh2015chiral}.

The index of Dirac operators on curved Lorentzian spacetimes seems not to have been discussed much in the physics or mathematics literature. 
We believe that this is due to the lack of ellipticity of the operator and the fact that in the hyperbolic case regularity is implied by propagation of singularities of the boundary data rather than the existence of a local parametrix. 
In \cite{finster2014chiral} an index is associated to a certain bounded operator acting on the solution space of the Dirac operator on some globally hyperbolic spacetimes.
There seems to be no relation to APS-boundary conditions and to the results of the present paper.
In \cite{gell2014feynman} a Fredholm property is established and a Feynman propagator is constructed for second-order wave equations in certain situations.

The paper is organized as follows:
In Section~\ref{sec:setup} we describe the setup and collect some standard material on Lorentzian geometry and on Dirac operators, mostly to fix notation.
In Section~\ref{sec:Cauchy} we introduce spinors of finite energy and discuss well-posedness of the Cauchy problem in those function spaces.
This gives rise to the wave evolution operator $Q$ for the Dirac equation.
The Atiyah-Patodi-Singer boundary conditions imply a decomposition of this evolution operator into a $2\times 2$-matrix.
We show that the off-diagonal terms $Q_{+-}$ and $Q_{-+}$ are compact operators while the diagonal terms $Q_{++}$ and $Q_{--}$ are Fredholm.
This is based on a Fourier integral operator technique.
Some technical material needed here is collected in Appendix~\ref{app:NormHyp}.
In Section~\ref{sec:DiracAPS} we relate the Fredholm property of $Q_{--}$ to that of the Dirac operator itself and show that the indices coincide.
Here we also show smoothness of solutions of the Dirac equation under APS or anti-APS boundary conditions.
The geometric formula for the index is computed in Section~\ref{sec:Geom}.
The spectral flow of a family of Dirac operators on Cauchy hypersurfaces is used to relate the index of the Lorentzian Dirac operator to that of a Dirac operator for an auxiliary Riemannian metric. 
In Section~\ref{sec:example} we construct Lorentzian metrics on $M=I\times S^{4k-1}$ with nontrivial index.

In Section~\ref{sec:Feynman} we describe an alternative approach to index theory for Dirac operators on Lorentzian manifolds based on the construction of a Feynman parametrix. We show that the Dirac operator subject to APS boundary conditions is invertible up to smoothing operators, at least in the case when the metric of $M$ is of product type near the boundary. 
This shows that the theory in some ways still resembles the elliptic theory, but Fourier integral operator parametrices replace the pseudodifferential parametrices. 
This approach also shows that the operators $Q_{+-}$ and $Q_{-+}$ are smoothing (and not just compact) when the metric of $M$ is of product type near the boundary.
This is important for physical applications as it implies the implementability of time evolution in the Fock space constructed from the space of solutions of the Dirac equation, by the Shale-Stinespring criterion.

We conclude in Section~\ref{sec:remarks} with some remarks on possible extensions of the results of this paper.

\emph{Acknowledgments.}
The authors would like to thank the \emph{Sonderforschungsbereich 647} funded by \emph{Deutsche Forschungsgemeinschaft},
and  the \emph{London Mathematical Society} (Scheme 4 grant 41325), for financial support and the universities of Loughborough 
and Potsdam for their hospitality. Part of the work for this paper was carried out during the trimester program
``Non-commutative Geometry and its Applications'', funded by the \emph{Hausdorff Center of Mathematics} in Bonn, and we are grateful for the support and hospitality during this time.


\section{Setup and notation}
\label{sec:setup}

We start by describing the setup of this article and collect a few standard facts on Dirac operators on Lorentzian manifolds.
For a more detailed introduction to Lorentzian geometry see e.g.\ \cite{beem1996global,o1983semi}.
Suppose that $X$ is an $(n+1)$-dimensional oriented time-oriented Lorentzian spin manifold. 
We use the convention that the metric of $X$ has signature $(-+\cdots +)$.

A subset $\Sigma\subset X$ is called a \emph{Cauchy hypersurface} if every inextensible timelike curve in $X$ meets $\Sigma$ exactly once.
If $X$ possesses a Cauchy hypersurface then $X$ is called \emph{globally hyperbolic}.
All Cauchy hypersurfaces of $X$ are homeomorphic.
We assume that $X$ is \emph{spatially compact}, i.e.\ the Cauchy hypersurfaces of $X$ are compact.
Let $\Sigma_-,\Sigma_+\subset X$ be two disjoint smooth and spacelike Cauchy hypersurfaces.
W.l.o.g.\ let $\Sigma_-$ lie in the past of $\Sigma_+$.
By \cite[Thm.~1.2]{MR2254187} (see also \cite[Thm.~1]{muller2015note}) $X$ can be written as 
\begin{equation}
X = \R \times \Sigma
\label{eq:Split}
\end{equation}
such that each $\Sigma_t = \{ t \} \times \Sigma$ is a smooth spacelike Cauchy hypersurface, $\Sigma_-=\Sigma_{t_1}$ and $\Sigma_+=\Sigma_{t_2}$. 
Moreover, the metric of $X$ takes the form $\<\cdot,\cdot\> = -N^2\, dt^2 + g_t$ where $N:X\to\R$ is a smooth positive function (the lapse function) and $g_t$ is a smooth $1$-parameter family of Riemannian metrics on $\Sigma$.
We understand the subset $M = [t_1,t_2] \times \Sigma$ as a globally hyperbolic manifold with boundary $\partial M =\Sigma_+ \sqcup \Sigma_-$.
Without reference to the splitting~\eqref{eq:Split}, $M$ can be characterized as $M=J^+(\Sigma_-)\cap J^-(\Sigma_+)$ where $J^\pm$ denotes the causal future and past, respectively.
Throughout the paper we will assume that $n$ is odd, i.e.\ the dimension of $M$ is even.

\ausblenden{Bild}{
\begin{pspicture}(-6.5,-3)(6,2.5)
\psset{viewpoint=-30 10 15, Decran=30, lightsrc=-20 20 15}
\defFunction{Lorentz}(u,v)
 {1 u u mul 0.1 mul add v Cos mul}
 {u}
 {1 u u mul 0.07 mul add v Sin mul}
\defFunction{Splus}(v)
 {1 -3 -3 mul 0.1 mul add v Cos mul}
 {-3}
 {1 -3 -3 mul 0.07 mul add v Sin mul}
\defFunction{Sminus}(v)
 {1 3 3 mul 0.1 mul add v Cos mul}
 {3}
 {1 3 3 mul 0.07 mul add v Sin mul}

\psSolid[object=surfaceparametree,
        base=-4 -3 0 2 pi mul,
        fillcolor=RoyalBlue!70,
        incolor=white,
	 opacity=0.7,
        function=Lorentz,
        ngrid=20 180,
        grid=false]%
\psSolid[object=surfaceparametree,
        base=-3 3 0 2 pi mul,
        fillcolor=blue!70,
        incolor=black,
	 opacity=0.7,
        function=Lorentz,
        ngrid=120 180,
        grid=false]%
\psSolid[object=surfaceparametree,
        base=3 4 0 2 pi mul,
        fillcolor=RoyalBlue!70,
        incolor=yellow!50,
	 opacity=0.7,
        function=Lorentz,
        ngrid=20 180,
        grid=false]%
\psSolid[object=courbe,
        range=0.95 4.22,
        r=0,
        ngrid=360,
        linecolor=black,
        linewidth=0.02,
        function=Splus]%
\psSolid[object=courbe,
        range=1.1 4.1,
        r=0,
        ngrid=360,
        linecolor=black,
        linewidth=0.02,
        function=Sminus]%
       
\psPoint(-3,0,1){M}
\uput[u](M){\psframebox*[framearc=.3]{$M$}}
\psPoint(-3,2.5,1){S-}
\uput[u](S-){\psframebox*[framearc=.3]{$\Sigma_-$}}
\psPoint(-3,-3.2,1){S+}
\uput[u](S+){\psframebox*[framearc=.3]{$\Sigma_+$}}
\end{pspicture}
\begin{center}
\textbf{Fig.~1.}
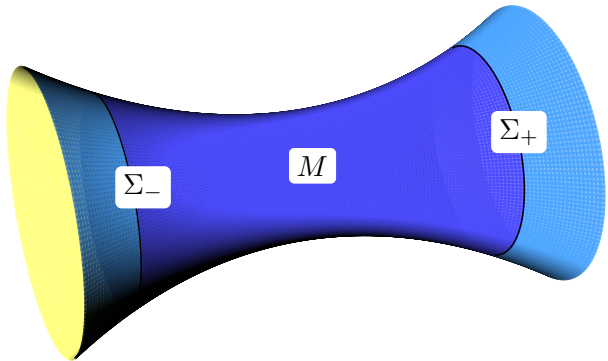
\emph{The Lorentzian manifold $X$}
\end{center}
} 

\subsection{Spinors and the Dirac operator}

We recall some facts about spinors and Dirac operators on Lorentzian manifolds, see \cite{MR2121740,MR701244} for details.
Let $SM \to M$ be the complex spinor bundle on $M$ endowed with its invariantly defined indefinite inner product $(\cdot,\cdot)$.
Suppose that $\D : C^\infty(M;SM) \to C^\infty(M;SM)$ is the Dirac operator acting on sections of $SM$. 
Locally, if $e_0,e_1,\ldots,e_n$ is a Lorentz-orthonormal tangent frame, the Dirac operator is given by
$$
\D= \sum_{j=0}^n \ep_j\gamma(e_j)\nabla_{e_j}
$$
where $\gamma(X)$ denotes Clifford multiplication by $X$, $\nabla$ is the Levi-Civita connection on $SM$, and $\ep_j=\<e_j,e_j\>=\pm 1$.
Clifford multiplication satisfies
$$
\gamma(X)\gamma(Y) + \gamma(Y)\gamma(X) = - 2 (X,Y)
$$
and 
$$
(\gamma(X) u,v) = (u,\gamma(X) v)
$$
for all $X,Y\in T_pM$, $u,v \in S_pM$ and $p\in M$.

Clifford multiplication with the volume form $\Gamma=\rmi^{n(n+3)/2}\,\gamma(e_0)\cdots\gamma(e_n)$ satisfies $\Gamma^2=\id_{SM}$.
This induces the eigenspace decomposition $SM = S^+M \oplus S^-M$ for the eigenvalues $\pm1$ into spinors of positive and negative chirality.
Since the dimension of $M$ is even, $\Gamma\gamma(X)=-\gamma(X)\Gamma$ for all $X\in TM$.
In particular, $S^+M$ and $S^-M$ have equal rank and the Dirac operator anticommutes with $\Gamma$, i.e.\ it takes the form 
$$
 \D = \left( \begin{matrix} 0 & \tilde D \\ D & 0 \end{matrix} \right)
$$
with respect to the splitting $SM = S^+M \oplus S^-M$.
Here $D : C^\infty(M;S^+M) \to C^\infty(M;S^-M)$ and  $\tilde D : C^\infty(M;S^-M) \to C^\infty(M;S^+M)$ are first order differential operators.
The subbundles $S^\pm M$ are isotropic with respect to the inner product $(\cdot,\cdot)$ of $SM$.

Let $\nu$ be the past-directed timelike vector field on $M$ with $\<\nu,\nu\>\equiv -1$ which is perpendicular to all $\Sigma_t$.
Then the divergence theorem implies for all $u,v \in C^1(M;SM)$
\begin{equation}
\int_M \{( \D u,v) + (u,\D v)\} \dV 
= 
\int_{\Sigma_{t_2}} (\gamma(\nu)u,v) \dA - \int_{\Sigma_{t_1}} (\gamma(\nu)u,v) \dA
\label{eq:DiracGreen}
\end{equation}
where $\dV$ denotes the volume element on $M$ and $\dA$ the one on $\Sigma_t$.
We denote the formal adjoint of any linear differential operator
$$
L: C^\infty(M;SM) \to C^\infty(M;SM)
$$
with respect to $(\cdot,\cdot)$ by $L^\dagger$.
If $u$ and $v$ are supported in the interior of $M$, then the boundary contribution in \eqref{eq:DiracGreen} vanishes which implies $\D^\dagger = -\D$. 

As is usual in the physics literature, we write $\beta=\gamma(\nu)$.
Then $\beta^2=\id$ and $\langle \cdot, \cdot \rangle=(\beta \cdot, \cdot)$ defines a positive definite inner product on $SM$ provided $(\cdot,\cdot)$ was chosen with the appropriate sign.
Note that in contrast to $(\cdot,\cdot)$ the inner product $\<\cdot,\cdot\>$ on $SM$ depends on the choice of splitting \eqref{eq:Split} which determines $\nu$ and hence $\beta$.

\subsection{Restriction to hypersurfaces}

The restriction of $S^\pm M$ to any slice $\Sigma_t$ can be naturally identified with the spinor bundle of $\Sigma_t$, i.e.\ $S^\pm M|_{\Sigma_t} = S\Sigma_t$, see \cite[Sec.~3]{MR2121740}.
On this restriction $\<\cdot,\cdot\>$ is now the natural inner product.
Clifford multiplication $\gamma_t(X)$ on $S\Sigma_t$ corresponds to $\rmi \beta \gamma(X)$ under this identification.
Since for any $X\in T\Sigma_t$
\begin{align*}
\<\gamma_t(X)u, v\> 
&= (\beta\rmi\beta\gamma(X)u, v) 
= (\rmi\gamma(X)u, v) 
= -(u,\rmi\gamma(X) v) \\
&= -(u,\beta\gamma_t(X) v) 
= -(\beta u,\gamma_t(X) v) 
= -\<u,\gamma_t(X) v\>
\end{align*}
Clifford multiplication on $\Sigma_t$ is skew-adjoint.
Moreover, under this identification $\beta$ anticommutes with $\gamma_t(X)$ and with the Dirac operator $A_t$ on $\Sigma_t$.
The latter can be seen as follows:
The spinorial Levi-Civita connections of $M$ and $\Sigma_t$ are related by
$$
\nabla_Xu = \nabla_X^{\Sigma_t}u - \frac12\beta\gamma(\nabla_X\nu)u
$$
for all $X$ tangent to $\Sigma_t$, see \cite[Eq.~(3.5)]{MR2121740}.
An easy computation now shows $\nabla_X^{\Sigma_t}(\beta u)=\beta\nabla_X^{\Sigma_t}u$.
Therefore, using a local orthonormal tangent frame $e_1,\ldots,e_n$ on $\Sigma_t$, we find 
$$
A_t(\beta u) 
= 
\sum_j\gamma_t(e_j)\nabla_{e_j}^{\Sigma_t}(\beta u)
=
\sum_j\rmi\beta\gamma(e_j)\beta\nabla_{e_j}^{\Sigma_t}u
=
-\sum_j\rmi\beta^2\gamma(e_j)\nabla_{e_j}^{\Sigma_t}u
=
-\beta A_tu.
$$
Along $\Sigma_t$ we have for $u \in C^\infty(M;S^\pm M)$
\begin{equation}
\D u = -\beta \left(\nabla_{\nu} + \rmi\, A_t - \frac{n}{2}H \right) u
\label{eq:DiracSurface+}
\end{equation}
where $H$ is the mean curvature of $\Sigma_t$ with respect to $\nu$, see \cite[Eq.~(3.6)]{MR2121740}.

If $M$ is a metric product near $\Sigma_{t_0}$, i.e.\ the metric takes the form $-dt^2 + g$ where $g$ is a Riemannian metric on $\Sigma$ independent of $t$, then $\nu = -\partial/\partial t$ (because $\nu$ is past-directed) and $H\equiv0$.
Denoting the Dirac operator on $(\Sigma,g)$ by $A$ and identifying spinors along the $t$-lines by parallel transport, then \eqref{eq:DiracSurface+} reduces near $\Sigma_{t_0}$ to
$$
 D =  \beta \left( \frac{\partial}{\partial t} - \mathrm{i}\, A \right) .
$$
Moreover, in the product case $\beta$ commutes with $\partial/\partial t$ and it always anticommutes with $A$, hence
$$
 \tilde D =   \left( \frac{\partial}{\partial t} + \mathrm{i}\, A \right) \beta.
$$

The second order differential operators $\tilde D D$ and $D \tilde D$ are normally hyperbolic.
In the product case they have the form $\frac{\partial^2}{\partial t^2} + A^2$.

\subsection{The Atiyah-Patodi-Singer spaces}

For any subset $I\subset\R$ denote by $\chi_I:\R\to\R$ the characteristic function of $I$. If $I$ is measurable we
denote the corresponding spectral projections of $A_t$ by $P_I(t):=\chi_I(A_t)$ and the corresponding subspaces of $L^2(\Sigma_t;SM|_{\Sigma_t})$ by $L^2_I(\Sigma_t;SM|_{\Sigma_t}):=P_I(t)(L^2(\Sigma_t;SM|_{\Sigma_t}))$.
Note that if $I$ is an interval then $P_I(t)$ is a zero-order pseudodifferential operator and it therefore acts
on the Sobolev spaces $H^s(\Sigma_t;SM|_{\Sigma_t})$ for any $s \in \R$.
For $s>1/2$, restriction to hypersurfaces is defined and continuous and we can therefore introduce the Atiyah-Patodi-Singer spaces as
\begin{gather*}
H^s_{\APS}( M;S^+M):=\{ u\in H^s(M;S^+M) \mid P_{[0,\infty)}(t_1)( u|_{\Sigma_{t_1}}) = 0 = P_{(-\infty,0]}(t_2)( u|_{\Sigma_{t_2}})\} \, ,\\
H^s_{\APS}(M;S^-M):=\{ u\in H^s(M;S^-M) \mid P_{(-\infty,0]}(t_1)( u|_{\Sigma_{t_1}}) = 0 = P_{[0,\infty)}(t_2)( u|_{\Sigma_{t_2}})\} \, .
\end{gather*}
We also denote $H^s_{\APS}( M;SM) := H^s_{\APS}( M;S^+M) \oplus H^s_{\APS}( M;S^-M) \subset H^s(M;SM)$.
Since $\beta$ anticommutes with $A_{t_j}$ it leaves $H^s_{\APS}(M;SM)$ invariant and continuously maps
$H^s_{\APS}( M;S^\pm M)$ to $H^s_{\APS}( M;S^\mp M)$.
Similarly, we define ``anti-Atiyah-Patodi-Singer'' spaces in such a way that their boundary values are orthogonal to the boundary
values of functions in $H^s_{\APS}$. This means
\begin{gather*}
H^s_{\aAPS}(M;S^+M):=\{ u\in H^s(M;S^+M) \mid P_{(-\infty,0)}(t_1)( u|_{\Sigma_{t_1}}) = 0 = P_{(0,\infty)}(t_2)( u|_{\Sigma_{t_2}})\} \,,\\
H^s_{\aAPS}( M;S^-M):=\{ u\in H^s(M;S^-M) \mid P_{(0,\infty)}(t_1)( u|_{\Sigma_{t_1}}) = 0 = P_{(-\infty,0)}(t_2)( u|_{\Sigma_{t_2}})\} \,,
\end{gather*}
and $H^s_{\aAPS}( M;SM) := H^s_{\aAPS}( M;S^+M) \oplus H^s_{\aAPS}( M;S^-M) \subset H^s(M;SM)$.


\section{The Cauchy problem}
\label{sec:Cauchy}

We recall that $M$ is foliated by the smooth spacelike Cauchy hypersurfaces $\Sigma_t$ where $t\in[t_1,t_2]$.
For any $s\in\R$ the family $\{H^s(\Sigma_t;S^+M|_{\Sigma_t})\}_{t\in[t_1,t_2]}$ is a bundle of Hilbert spaces over the interval $[t_1,t_2]$.
As a bundle of topological vector spaces, it can be globally trivialized by parallel transport along the $t$-lines, for example.
Continuous sections of this bundle will be called \emph{spinors of finite $s$-energy} and the space of all such sections will be denoted by $\FE^s(M;S^+M)$.
Any spinor $u$ of finite $s$-energy can naturally be considered as a distributional spinor on $M$ via
$$
u[\vp] = \int_{t_1}^{t_2} u(t)[(N\vp)|_{\Sigma_t}]\, \dt
$$ 
for all test sections $\vp\in C^\infty(M;(S^+)^*M)$.
Here $ u(t)[(N\vp)|_{\Sigma_t}]$ refers to the distributional application of $u(t)\in H^s(\Sigma_t;S^+M|_{\Sigma_t})$ to the test spinor $(N\vp)|_{\Sigma_t}\in C^\infty(\Sigma_t;(S^+)^*M|_{\Sigma_t})$.
Note that the volume element of $M$ is $N\,\dt\,\dA$ which explains the appearance of the lapse function $N$ as a factor in the above formula.
The space $\FE^s(M;S^+M)$ is topologized by the norm
$$
\|u\|_{\FE^s} = \max_{[t_1,t_2]} \|u(t)\|_{H^s}.
$$
For $s>\frac{n}{2}$ the Sobolev embedding theorem implies $\FE^s(M;S^+M)\subset C^0(M;S^+M)$.

In order to treat the inhomogeneous Cauchy problem for the Dirac operator we denote the space of $L^2$-sections of the bundle $\{H^s(\Sigma_t;S^+M|_{\Sigma_t})\}_{t\in[t_1,t_2]}$ by $L^2([t_1,t_2];H^s(\Sigma_\bullet))$ and equip it with the corresponding $L^2$-norm
$$
\|u\|^2_{L^2,H^s} := \int_{t_1}^{t_2} \|(Nu)|_{\Sigma_t}\|^2_{H^s}\, \dt \, .
$$
Now we define
$$
\FE^s(M;D) := \{u\in\FE^s(M;S^+M)\mid Du\in L^2([t_1,t_2];H^s(\Sigma_\bullet))\}
$$
with the norm
$$
\|u\|^2_{\FE^s,D} := \|u\|^2_{\FE^s} + \|Du\|^2_{L^2,H^s} \,\,.
$$
Here $D$ is applied to $u$ in the distributional sense.
The following statement is known as well-posedness of the inhomogeneous Cauchy problem for the Dirac equation:

\begin{theorem}\label{thm:InhomoCauchyProblem}
For any $s\in\R$ and any $t\in[t_1,t_2]$ the mapping
\begin{align*}
\res_t \oplus D : \FE^s(M;D) &\to H^s(\Sigma_t;S^+M|_{\Sigma_t})\oplus L^2([t_1,t_2];H^s(\Sigma_\bullet)),\\
u &\mapsto (u|_{\Sigma_t},Du),
\end{align*}
is an isomorphism of Banach spaces.
\end{theorem}

\begin{proof}
Recall that $D=\frac{\beta}{N} \left(\frac{\partial}{\partial t} - \rmi\,N\, A_t + \frac{n}{2}NH \right)$.
Applying Theorem~3.2 in \cite[Ch.~IV]{MR618463} with $K=\rmi\,N\,  A_t - \frac{n}{2}NH$ gives that the map is bijective.
It is clearly continuous.
By the open mapping theorem it is an isomorphism.
\end{proof}

\begin{rems}
\begin{enumerate}[label=(\alph*),leftmargin=0pt,labelindent=15pt,labelsep=2pt,labelwidth=5pt,itemindent=!]
\item
Since $H^s(\Sigma_t;S^+M|_{\Sigma_t})\oplus L^2([t_1,t_2];H^s(\Sigma_\bullet))$ is a Hilbert space we conclude that $\FE^s(M;D)$ also carries a Hilbert space topology.
Note that $\FE^s(M;S^+M)$ is only a Banach space.
\item
Smooth sections are dense in $H^s(\Sigma_t;S^+M|_{\Sigma_t})$ and in $L^2([t_1,t_2];H^s(\Sigma_\bullet))$.
Since the solution of the initial value problem $Du=f$, $u|_{\Sigma_t}=u_0$, is smooth if $f$ and $u_0$ are smooth (cf.\ the proof of Theorem~\ref{thm:KerDAPSglatt}), we get that smooth sections are also dense in $\FE^s(M;D)$.
\item
For general $s$ the spaces $\FE^s(M;D)$ may depend on the choice of the foliation of $M$ into Cauchy hypersurfaces.
For $s=0$ however, the right hand side in Theorem~\ref{thm:InhomoCauchyProblem} is 
$$
L^2(\Sigma_t;S^+M|_{\Sigma_t}) \oplus L^2([t_1,t_2];L^2(\Sigma_\bullet))
=
L^2(\Sigma_t;S^+M|_{\Sigma_t}) \oplus L^2(M;S^-M)
$$
and hence independent of this choice.
Thus $\FE^0(M;D)$ is the completion of $C^\infty(M;S^+M)$ with respect to a norm which does not depend on the choice of foliation.
Hence $\FE^0(M;D)$ is intrinsically defined by $M$.
\item
The space $\FE^0(M;D)$ is the completion of $C^\infty(M;S^+M)$ with respect to the norm 
$$
\|u\|^2_{\FE^0,D} = \max_{\tau\in[t_1,t_2]}\|u|_{\Sigma_\tau}\|^2_{L^2} + \|Du\|^2_{L^2} \, .
$$
By the previous remark the norm
$$
\|u|_{\Sigma_t}\|^2_{L^2} + \|Du\|^2_{L^2}
$$
for fixed $t\in[t_1,t_2]$ is an equivalent norm on $\FE^0(M;D)$.
This is not surprising because the ``energy estimate'' (2.5) in \cite[Ch.~IV]{MR618463} ensures existence of a constant $C$ such that 
$$
\|u|_{\Sigma_t}\|^2_{L^2}
\le
C\cdot(\|u|_{\Sigma_\tau}\|^2_{L^2} + \|Du\|^2_{L^2})
$$
for any $t,\tau\in[t_1,t_2]$.
Integration with respect to $\tau$ yields
$$
(t_2-t_1)\|u|_{\Sigma_t}\|^2_{L^2}
\le
C\cdot(\|u\|^2_{L^2} + (t_2-t_1)\|Du\|^2_{L^2}) \, .
$$
On the other hand, we clearly have $\|u\|^2_{L^2}\le C'\max_{\tau\in[t_1,t_2]}\|u|_{\Sigma_\tau}\|^2_{L^2}$.
Therefore the $L^2$-graph norm for $D$ given by $\|u\|^2_{L^2} + \|Du\|^2_{L^2}$ is also an equivalent norm on $\FE^0(M;D)$.
\end{enumerate}
\end{rems}

Both $\FE^s(M;S^+M)$ and $\FE^s(M;D)$ induce the same relative topology on 
$$
\FE^s(M;\ker(D)) := \{u\in\FE^s(M;S^+M)\mid Du=0\} \, .
$$

We get well-posedness of the homogeneous Cauchy problem for the Dirac equation:

\begin{cor}\label{cor:HomoCauchyProblem}
For any $t\in[t_1,t_2]$ the restriction mapping
$$
\res_t : \FE^s(M;\ker(D)) \to H^s(\Sigma_t;S^+M|_{\Sigma_t})
$$
is an isomorphism of topological vector spaces.\qed
\end{cor}

For $t,t'\in [t_1,t_2]$ we define the \emph{wave evolution operator} 
$$
Q(t',t):H^s(\Sigma_t;S^+M|_{\Sigma_t}) \to H^s(\Sigma_{t'};S^+M|_{\Sigma_{t'}})
$$ 
by the commutative diagram
$$
\xymatrix{
&  \FE^s(M;\ker(D)) \ar[dr]^{\res_{t'}}_\cong \ar[dl]_{\res_{t}}^\cong &  \\
H^s(\Sigma_t;S^+M|_{\Sigma_t}) \ar[rr]^{Q(t',t)} & & H^s(\Sigma_{t'};S^+M|_{\Sigma_{t'}})
}
$$

\subsection{Properties of the evolution operator}

It is clear from the definition that for each $s\in\R$ the operator $Q(t',t)$ is an isomorphism of topological vector spaces and that $Q(t'',t')\circ Q(t',t) = Q(t'',t)$ holds for all $t,t',t''\in [t_1,t_2]$.
In particular, $Q(t,t)=\id_{H^s(\Sigma_t)}$ and $Q(t,t')=Q(t',t)^{-1}$.

\begin{lem}\label{lem:unitaer}
For $s=0$ and any $t,t'\in [t_1,t_2]$ the wave evolution operator is unitary, i.e.\ an isometry
$$
Q(t',t):L^2(\Sigma_t;S^+M|_{\Sigma_t}) \to L^2(\Sigma_{t'};S^+M|_{\Sigma_{t'}}) \,\,.
$$
\end{lem}
\begin{proof}
It suffices to check $\|Q(t',t) u\|_{L^2(\Sigma_{t'})} = \| u\|_{L^2(\Sigma_{t})}$ for $ u$ in a dense subset of $L^2(\Sigma_{t};S^+M|_{\Sigma_{t}})$.
Fix $s>\frac{n}{2} + 2$ and let $ u\in H^s(\Sigma_{t};S^+M|_{\Sigma_{t}})$.
By Corollary~\ref{cor:HomoCauchyProblem} there is a unique $\Phi\in\FE^s(M;\ker(D))$ which restricts to $ u$.
Since $D\Phi=0$ equation~\eqref{eq:DiracSurface+} shows $\nabla_t\Phi = -N\nabla_\nu\Phi = N(\rmi A_t-\frac{n}{2}H)\Phi$ and hence $\nabla_t\Phi\in\FE^{s-1}(M;S^+M)$.
Therefore $\Phi$ is a $C^1$-section of $\{H^{s-1}(\Sigma_t;S^+M|_{\Sigma_t})\}_{t\in[t_1,t_2]}$.
By the Sobolev embedding theorem $H^{s-1}(\Sigma_t;S^+M|_{\Sigma_t}) \subset C^1(\Sigma_t;S^+M|_{\Sigma_t})$.
Hence $\Phi\in C^1(M;S^+M)$.

Thus \eqref{eq:DiracGreen} applies to $\Phi$ on $[t,t']\times\Sigma\subset M$ (w.l.o.g.\ assume $t'>t$) and gives
\begin{align*}
0
&=
\int_M \{(D\Phi,\Phi) + (\Phi,D\Phi)\} \dV \\
&= 
\int_{\Sigma_{t'}} (\beta Q(t',t) u,Q(t',t) u) \dA - \int_{\Sigma_{t}} (\beta u, u) \dA \\
&=
\|Q(t',t) u\|^2_{L^2(\Sigma_{t'})} - \| u\|^2_{L^2(\Sigma_{t})} \, .\qedhere
\end{align*}
\end{proof}

With respect to the $L^2$-orthogonal splittings 
\begin{align*}
L^2(\Sigma_{t_1};S^+M|_{\Sigma_{t_1}}) 
&= 
L^2_{[0,\infty)}(\Sigma_{t_1};S^+M|_{\Sigma_{t_1}}) \oplus L^2_{(-\infty,0)}(\Sigma_{t_1};S^+M|_{\Sigma_{t_1}}) \, , \\
L^2(\Sigma_{t_2};S^+M|_{\Sigma_{t_2}}) 
&= 
L^2_{(0,\infty)}(\Sigma_{t_2};S^+M|_{\Sigma_{t_2}}) \oplus L^2_{(-\infty,0]}(\Sigma_{t_2};S^+M|_{\Sigma_{t_2}}) \, , 
\end{align*}
we write $Q(t_2,t_1)$ as a $2\times 2$-matrix,
\begin{equation}
Q(t_2,t_1) =
\begin{pmatrix}
\Qpp & \Qpm \\
\Qmp & \Qmm
\end{pmatrix} \, .
\label{eq:QMatrix}
\end{equation}
Hence $\Qmm = P_{(-\infty,0]}(t_2) \circ Q(t_2,t_1)|_{L^2_{(-\infty,0)}(\Sigma_{t_1};S^+M|_{\Sigma_{t_1}})}$ and similarly for the other three entries in the matrix.

\begin{lem}\label{lem:kerQ}
The operator $\Qpm$ restricts to an isomorphism 
$$
\ker[\Qmm]\to \ker[\Qpp^*]
$$ 
and $\Qmp$ restricts to an isomorphism 
$$
\ker[\Qpp]\to \ker[\Qmm^*].
$$
\end{lem}

\begin{proof}
Since $Q(t_2,t_1)$ is unitary we have $Q(t_2,t_1)^*Q(t_2,t_1)=\id$.
Spelled out in terms of the matrix entries this means
\begin{align}
\Qpp^*\Qpp + \Qmp^*\Qmp &= \id, \label{eq:Q*Q1}\\
\Qmm^*\Qmm + \Qpm^*\Qpm &= \id, \label{eq:Q*Q2}\\
\Qmp^*\Qmm + \Qpp^*\Qpm &= 0,   \label{eq:Q*Q3}\\
\Qpm^*\Qpp + \Qmm^*\Qmp &= 0.   \label{eq:Q*Q4}
\end{align}
If $u\in\ker[\Qmm]$ then by \eqref{eq:Q*Q3} 
$$
\Qpp^*\Qpm u = - \Qmp^*\Qmm u=0.
$$
Hence $\Qpm u\in\ker[\Qpp^*]$, i.e.\ $\Qpm$ restricts to a map $\ker[\Qmm]\to \ker[\Qpp^*]$. 
Moreover, by \eqref{eq:Q*Q2}, 
$$
u = \Qmm^*\Qmm u + \Qpm^*\Qpm u = \Qpm^*\Qpm u.
$$
Similarly, using  $Q(t_2,t_1)Q(t_2,t_1)^*=\id$, one sees that $\Qpm^*$ restricts to a map $\ker[\Qpp^*]\to \ker[\Qmm]$ and that for $u\in\ker[\Qpp^*]$ we have $\Qpm\Qpm^*u=u$.
Thus $\Qpm^*:\ker[\Qpp^*]\to \ker[\Qmm]$ is the inverse of $\Qpm:\ker[\Qmm]\to \ker[\Qpp^*]$.

The proof of the second statement is analogous.
\end{proof}

\begin{lem} \label{FIOlemma}
 The operators $\Qpm \circ P_{(-\infty,0)}(t_1)$ and $\Qmp \circ P_{(0,\infty)}(t_1)$ map $H^s(\Sigma_{t_1};S^+M|_{\Sigma_{t_1}})$ continuously to $H^{s+1}(\Sigma_{t_2};S^+M|_{\Sigma_{t_2}})$. 
 In particular, $\Qpm$ and $\Qmp$ are compact as operators from $L^2_{(-\infty,0]}(\Sigma_{t_1};S^+M|_{\Sigma_{t_1}})\to L^2(\Sigma_{t_2};S^+M|_{\Sigma_{t_2}})$, and $L^2_{(0,\infty)}(\Sigma_{t_1};S^+M|_{\Sigma_{t_1}})\to L^2(\Sigma_{t_2};S^+M|_{\Sigma_{t_2}})$, respectively.
 Similarly, $\Qmp^*\circ P_{(-\infty,0)}(t_2)$ and $\Qpm^* \circ P_{(0,\infty)}(t_2)$ map $H^s(\Sigma_{t_2};S^+M|_{\Sigma_{t_1}})$ continuously to $H^{s+1}(\Sigma_{t_1};S^+M|_{\Sigma_{t_1}})$. 
\end{lem}

\begin{proof}
 We show this only for $\Qpm$ and its adjoint.
 The proof for $\Qmp$ is analogous. 
 We will need a description of the principal symbol of $Q(t_2,t_2)$ as a Fourier integral operator. 
 The operator $D \tilde D$ is a normally hyperbolic operator on the globally hyperbolic manifold $M$ and, by Theorem~\ref{fioth}, the solution operator 
 $$
\T_t: C^\infty(\Sigma_{t};S^-M|_{\Sigma_{t}}) \to C^\infty(M;S^-M), \quad f \mapsto u,
 $$
 of the initial value problem
 $$
  D \tilde D u =0, \quad u|_{\Sigma_{t}}=0, \quad (-\nabla_\nu u)|_{\Sigma_{t}}=f,
 $$
 is therefore a Fourier integral operator of order $-\frac{5}{4}$ with principal symbol given by (\ref{princsym}) and canonical relation $C_1$ as described in Theorem~\ref{fioth}.
 The map $\tilde D \circ \T_{t} \circ \beta: f \mapsto u$ solves the Cauchy problem
  $$
  D u =0, \quad u|_{\Sigma_{t}}=f,
 $$
 and by the above it is a Fourier integral operator of order $-\frac{1}{4}$ and canonical relation $C_1$.
We obtain $Q(t_2,t_1) = \res_{t_2} \circ \tilde D \circ \T_{t_1} \circ \beta$.
 The restriction operator $\res_{t_2}$ is a Fourier integral operator of order $\frac{1}{4}$ and the canonical relation of $\res_{t_2}$ is
 $$C_2 =\{((y,\eta),(x,\xi)) \in  T^* \Sigma_{t_2}   \times T^* M \mid (x,\xi) \in \dot T^* M, \; \res_{t_2}^* (x,\xi) = (y,\eta)\}\;.$$
 
 The composition $C=C_2 \circ C_1$ is thus described as follows. 
 If $(y,\eta)$ is in $\dot T^*\Sigma_{t_1}$ then there exist two lightlike covectors $(y,\tilde \eta_+), (y,\tilde\eta_-)\in T^*M$ that restrict to $(y,\eta)$.
  For the sake of definiteness we choose $\tilde\eta_+$ to be future directed and $\tilde\eta_-$ to be past directed.
  The orbit of  $(y,\tilde\eta_\pm)$ under the geodesic flow intersects $T^*M|_{\Sigma_{t_2}}$ at precisely one point
 $(x_\pm,\tilde \xi_\pm)$. Let $(x_\pm, \xi_\pm)$ be the pull back of these co-vectors to $T^* \Sigma_{t_2}$.
 The canonical relation $C$ therefore relates  $(y,\eta)$ to the two points $(x_+, \xi_+)$ and $(x_-, \xi_-)$.
 Since $M$ is globally hyperbolic the composition $C_2 \circ C_1$  is proper and transversal (see the discussion in Section 5.1
 of \cite{Duistermaat:1996aa}).
 This implies that
  $\res_{t_2} \circ \tilde D \circ \T_{t_1} \circ \beta$ is Fourier integral operator of order zero (see for example  \cite[Thm.~25.2.3]{Hormander:1985aaiv})
  with canonical relation $C$.
 In particular, $\res_{t_2} \circ \tilde D \circ \T_{t_1}\circ \beta$ is a sum of two Fourier integral operators
 $F_+$ and $F_-$, such that $F_\pm$ is associated to a canonical map $(y,\eta) \mapsto (x,\xi_\pm)$.
 The principal symbols of these Fourier integral operators can be computed using \cite[Thm.~25.2.3]{Hormander:1985aaiv} or \cite[Thm.~4.2.2]{Duistermaat:1996aa} and are given by
 $$
  \pm \frac{1}{2} \| \eta \|^{-1} \left (-\tilde \xi_\pm(\nu)  \beta + \gamma(\xi_\pm) \right) \circ \Gamma_{(x,\tilde \xi_\pm),(y,\tilde\eta_\pm) } \circ \beta,
 $$
 where $\Gamma_{(x,\tilde \xi_\pm),(y,\tilde\eta_\pm) }$ is the operator of parallel transport from $(y,\tilde\eta_\pm)$ to $(x,\tilde \xi_\pm)$, and $\gamma(\xi_\pm)=-\rmi \beta \gamma_{t_2}(\xi_\pm)$.
 Note that the projections $P_{(0,\infty)}(t_2)$ and $P_{(-\infty,0)}(t_1)$ are pseudodifferential operators
 of order zero acting on $L^2(\Sigma_{t_2};S^+M|_{\Sigma_{t_2}})$. Their principal symbols $\sigma_{P_{(0,\infty)}(t_2)}$ 
 and $\sigma_{P_{(-\infty,0)}(t_1)}$ 
 can be computed from the principal symbol $\sigma_{A_t}$ of the operator $A_t$, which is easily found to be $\sigma_{A_t}(\eta)=-\beta \gamma(\eta)=\rmi \gamma_t(\eta)$.
 One obtains
 \begin{align*}
  \sigma_{P_{(0,\infty)}(t_2)}(y,\eta) &= \frac{1}{2}\left(\id - \beta \| \eta \|^{-1} \gamma(\eta)\right),\\
  \sigma_{P_{(-\infty,0)}(t_1)}(y,\eta) &= \frac{1}{2}\left(\id+\beta \| \eta \|^{-1} \gamma(\eta)\right).
 \end{align*}
 Therefore the operator $\Qpm \circ P_{(-\infty,0)}(t_1)$ is a sum of two Fourier integral operators associated to the canonical maps
 $(y,\eta) \mapsto (x,\xi_\pm)$with principal symbols
 \begin{align*}
 q_\pm(y,\eta)=\pm \frac{1}{8}\| \xi_\pm \|^{-1} \left( \| \xi_\pm \|  - \beta \gamma(\xi_\pm) \right) \left(- \tilde \xi_\pm(\nu) \beta + \gamma(\xi_\pm)\right)  
 \\ \circ  \Gamma_{(x,\tilde \xi_\pm),(y,\tilde\eta_\pm) }  \circ \| \eta \|^{-2} \beta (\| \eta \| + \beta \gamma(\eta)).
 \end{align*}
Since $\tilde \xi_\pm$ is lightlike we have $\tilde \xi_\pm(\nu)=  \mp \| \xi_\pm\|$. 
Thus
$$
 ( \| \xi_- \|  - \beta \gamma(\xi_-)) ( -\tilde \xi_-(\nu) \beta + \gamma(\xi_-))=- \left( \| \xi_- \|  - \beta \gamma(\xi_-)\right)  
 \left ( \| \xi_- \|  + \beta \gamma(\xi_-)\right) \beta=0,$$
 and consequently $q_-(y,\eta)=0$. We have used here that $\beta$ anticommutes with $\gamma(\xi_-)$ and $(\beta \gamma(\xi_-))^2  = \|\xi_-\|^2$.
Clifford multiplication by $\tilde \xi_\pm$ is given by $(\pm \| \xi_\pm \| \beta + \gamma(\xi_\pm))$
and because parallel transport is compatible with Clifford multiplication we have
 $$
 \Gamma_{(x,\tilde \xi_\pm),(y,\tilde \eta_\pm) } \circ (\pm \| \eta \| \beta+\gamma(\eta)) = (\pm \| \xi_\pm \| \beta + \gamma(\xi_\pm)) \circ  
 \Gamma_{(x,\tilde \xi_\pm),(y,\tilde \eta_\pm) }.
 $$
Thus
\begin{align*}
q_+(y,\eta) 
&=
 +\frac{1}{4} \left ( \beta \| \xi_+ \|  +  \gamma(\xi_+) \right) \circ  \Gamma_{(x,\tilde \xi_+),(y,\tilde\eta_+) } 
\circ \| \eta \|^{-2}\left(\| \eta \| - \beta \gamma(\eta)\right) \beta \\ 
&=
\frac{1}{4} \| \eta \|^{-2} \Gamma_{(x,\tilde \xi_+),(y,\tilde\eta_+)} \circ  \beta \left (\| \eta \| + \beta \gamma(\eta) \right) 
\left(\| \eta \| - \beta\gamma(\eta)\right) \beta\\
&=0.
\end{align*}
We have shown that the principal symbols of the zero order operators vanish and therefore $\Qpm \circ P_{(-\infty,0)}(t_1)$ is the sum of two Fourier integral operators of order $-1$, each of them associated to a canonical map.
The statement of the theorem now follows from the mapping properties of Fourier integral operators whose canonical relation is a canonical graph
(see for example \cite[Cor.~25.3.2]{Hormander:1985aaiv}).
 \end{proof}

\begin{cor}\label{cor:Qglatt}
The kernel of $\Qmm$, and $\Qpp$ respectively, consists of smooth sections.
\end{cor}
\begin{proof}
Equation~\eqref{eq:Q*Q2} implies that every element in the kernel of $\Qmm$ is in the $(+1)$-eigenspace of $\Qpm^*\Qpm$. 
Lemma~\ref{FIOlemma} then shows that this eigenspace consists of functions with Sobolev regularity of any degree.
The same argument applies to the kernel of $\Qpp$.
\end{proof}

\section{The Dirac operator with Atiyah-Patodi-Singer boundary conditions}
\label{sec:DiracAPS}

In this section we relate the Fredholm property of $\Qmm$ to that of the Dirac operator itself under Atiyah-Patodi-Singer boundary conditions.
The discussion is based on the following auxiliary lemma.
For a proof see e.g.\ \cite[Prop.~A.1]{MR3076058}.

\begin{lem}
\label{lem:Fredholm}
Let $H$ be a Hilbert space, let $E$ and $F$ be Banach spaces and let $L:H \to E$ and $P:H \to F$ be bounded linear maps.
We assume that $P:H \to F$ is onto.

Then $L|_{\ker[P]}: \ker[P] \to E$ is Fredholm of index $k$ if and only if $L\oplus P: H \to E\oplus F$ is Fredholm of index $k$.\qed
\end{lem}

\begin{theorem}\label{thm:IndQ=IndD}
The operator 
\begin{align*}
\left(P_{[0,\infty)}(t_1)\circ\res_{t_1}\right) &\oplus \left( P_{(-\infty,0]}(t_2)\circ\res_{t_2}\right) \oplus D: \\
\FE^0(M;D) &\to L^2_{[0,\infty)}(\Sigma_{t_1};S^+M|_{\Sigma_{t_1}}) \oplus L^2_{(-\infty,0]}(\Sigma_{t_2};S^+M|_{\Sigma_{t_2}}) \oplus L^2(M;S^-M)
\end{align*}
is Fredholm and its index satisfies
$$
\ind[\left(P_{[0,\infty)}(t_1)\circ\res_{t_1}\right) \oplus \left( P_{(-\infty,0]}(t_2)\circ\res_{t_2}\right)\oplus D]
= \ind[\Qmm] \, .
$$
\end{theorem}

\begin{proof}
Theorem~\ref{thm:InhomoCauchyProblem} with $s=0$ shows that $\FE^0(M;D)$ carries the topology of a Hilbert space and that 
$$
D:\FE^0(M;D) \to  L^2([t_1,t_2],L^2(\Sigma_\bullet)) = L^2(M;S^-M)
$$
is onto.
We can therefore apply Lemma~\ref{lem:Fredholm} with 
\begin{align*}
H&=\FE^0(M;D),\\
E&=L^2_{[0,\infty)}(\Sigma_{t_1};S^+M|_{\Sigma_{t_1}}) \oplus L^2_{(-\infty,0]}(\Sigma_{t_2};S^+M|_{\Sigma_{t_2}}),\\
F&=L^2(M;S^-M),\\
L&= \left(P_{[0,\infty)}(t_1)\circ\res_{t_1}\right) \oplus \left( P_{(-\infty,0]}(t_2)\circ\res_{t_2}\right), \mbox{ and}\\
P&=D.
\end{align*}
It remains to check that the operator $L_0:=L|_{\ker(D)}:\FE^0(M;\ker(D))\to E$ is Fredholm with the same index as $\Qmm$.
As to the kernel we have
\begin{align*}
\ker[L_0] 
&=
\{u\in\FE^0(M;\ker(D)) \mid u|_{\Sigma_{t_1}} \in L^2_{(-\infty,0)}(\Sigma_{t_1};S^+M|_{\Sigma_{t_1}}),\, P_{(-\infty,0]}(t_2)(u|_{\Sigma_{t_2}})=0\}\\
&\cong
\ker[\Qmm]\, .
\end{align*}
The image is given by
\begin{align*}
\mathrm{im}(L_0)
&=
\{(u_1,u_2) \mid u_2 = \Qmp u_1 + \Qmm w \mbox{ for some }w\in L^2_{(-\infty,0)}(\Sigma_{t_1})\} \\
&=
\{(u_1,\Qmp u_1 + \Qmm w ) \mid u_1\in L^2_{[0,\infty)}(\Sigma_{t_1}), \,w\in L^2_{(-\infty,0)}(\Sigma_{t_1})\} \, .
\end{align*}
This shows in particular that the image is closed.
Namely, let $u_{1,i}\in L^2_{[0,\infty)}(\Sigma_{t_1};S^+M|_{\Sigma_{t_1}})$ and $w_i\in L^2_{(-\infty,0)}(\Sigma_{t_1};S^+M|_{\Sigma_{t_1}})$ such that $(u_{1,i},\Qmp u_{1,i} + \Qmm w_i )\to(u_1,u_2)$ in $E$.
Then $\Qmp u_{1,i}\to\Qmp u_1$ and hence $\Qmm w_i \to u_2 -\Qmp u_1$.
Since $\Qmm$ is Fredholm its image is closed and thus $u_2 -\Qmp u_1 = \Qmm w$ for some $w\in L^2_{(-\infty,0)}(\Sigma_{t_1};S^+M|_{\Sigma_{t_1}})$.
Therefore $(u_1,u_2)=(u_1,\Qmp u_1 + \Qmm w)$ lies in $\mathrm{im}(L_0)$.

\reqnomode
Moreover,
\begin{alignat}{3}
\mathrm{im}(L_0)^\perp
&=
\{(v_1,v_2)\in E \mid && (v_1,u_1)_{L^2} + (v_2,\Qmp u_1 + \Qmm w)_{L^2} = 0 & \notag\\
& &&\mbox{for all } u_1\in L^2_{[0,\infty)}(\Sigma_{t_1}), \,w\in L^2_{(-\infty,0)}(\Sigma_{t_1})\} \notag\\
&=
\{(v_1,v_2)\in E \mid && (v_1+\Qmp^* v_2,u_1)_{L^2} + (\Qmm^* v_2, w)_{L^2} = 0 & \notag\\
& &&\mbox{for all } u_1\in L^2_{[0,\infty)}(\Sigma_{t_1}), \,w\in L^2_{(-\infty,0)}(\Sigma_{t_1})\} \notag\\
&=
\{(v_1,v_2)\in E \mid &&  v_1+\Qmp^* v_2=0 \mbox{ and } \Qmm^* v_2=0  \}  & \notag\\
&=
\{(-\Qmp &&^* v_2, v_2) \mid v_2\in\ker[\Qmm^*] \} & \notag\\
&\cong
\ker[\Qmm &&^*]. &  \tag*{\qed}
\end{alignat}
\phantom\qedhere
\end{proof}
\leqnomode

We now consider finite energy spinors which satisfy the Atiyah-Patodi-Singer boundary conditions and define
\begin{equation*}
\FE^s_\APS(M;D) := \{u\in \FE^s(M;D) \mid P_{[0,\infty)}(t_1)(u|_{\Sigma_{t_1}})=0=P_{(-\infty,0]}(t_2)(u|_{\Sigma_{t_2}})\} \, .
\end{equation*}
The space $\FE^s_\APS(M;D)$ is a closed subspace of $\FE^s(M;D)$ and carries the relative topology.

\begin{theorem}\label{thm:IndQ=IndDAPS}
The operator 
$$
D_\APS := D|_{\FE^0_\APS} : \FE^0_\APS(M;D) \to L^2(M;S^-M)
$$
is Fredholm and its index satisfies
$$
\ind[D_\APS] = \ind[\Qmm] \, .
$$
\end{theorem}

\begin{proof}
This time we apply Lemma~\ref{lem:Fredholm} with
\begin{align*}
H&=\FE^0(M;D),\\
E&=L^2(M;S^-M),\\
F&=L^2_{[0,\infty)}(\Sigma_{t_1};S^+M|_{\Sigma_{t_1}}) \oplus L^2_{(-\infty,0]}(\Sigma_{t_2};S^+M|_{\Sigma_{t_2}}),\\
L&=D,\mbox{ and}\\
P&= \left(P_{[0,\infty)}(t_1)\circ\res_{t_1}\right) \oplus \left( P_{(-\infty,0]}(t_2)\circ\res_{t_2}\right).
\end{align*}
Then we get that $D_\APS$ is Fredholm with $\ind[D_\APS]=\ind[D\oplus P]$.
Theorem~\ref{thm:IndQ=IndD} concludes the proof.
\end{proof}

Looking at finite energy spinors satisfying anti-Atiyah-Patodi-Singer boundary conditions
\begin{equation*}
\FE^s_\aAPS(M;D) := \{u\in \FE^s(M;D) \mid P_{(-\infty,0)}(t_1)(u|_{\Sigma_{t_1}})=0=P_{(0,\infty)}(t_2)(u|_{\Sigma_{t_2}}) \}
\end{equation*}
the same reasoning shows

\reqnomode
\begin{theorem}
The operators
\begin{align*}
\left(P_{(-\infty,0)}(t_1)\circ\res_{t_1}\right) &\oplus \left( P_{(0,\infty)}(t_2)\circ\res_{t_2}\right) \oplus D: \\
\FE^0(M;D) &\to L^2_{(-\infty,0)}(\Sigma_{t_1};S^+M|_{\Sigma_{t_1}}) \oplus L^2_{(0,\infty)}(\Sigma_{t_2};S^+M|_{\Sigma_{t_2}}) \oplus L^2(M;S^-M)
\end{align*}
and
$$
D_\aAPS := D|_{\FE^0_\aAPS} : \FE^0_\aAPS(M;D) \to L^2(M;S^-M)
$$
are Fredholm and their indices satisfy
\begin{align}
\ind[\left(P_{[0,\infty)}(t_1)\circ\res_{t_1}\right) \oplus \left( P_{(-\infty,0]}(t_2)\circ\res_{t_2}\right)\oplus D]
=
\ind[D_\aAPS]
= 
\ind[\Qpp] \,\,.
\tag*{\qed}
\end{align}
\end{theorem}
\leqnomode

In the Riemannian setting Dirac operators are elliptic and hence the kernels of Dirac operators consist of smooth spinors, independently of any boundary conditions.
In our Lorentzian situation this is no longer true. 
Solutions to the Dirac equation can have low regularity.
Remarkably, under APS or aAPS boundary conditions solutions are again smooth just as if we had elliptic regularity theory at our disposal.

\begin{theorem}\label{thm:KerDAPSglatt}
The kernels of the operators $D_\APS: \FE^0_\APS(M;S^+M) \to  L^2(M;S^-M)$ and $D_\aAPS: \FE^0_\aAPS(M;S^+M) \to  L^2(M;S^-M)$ consist of smooth spinors.
The indices satisfy
$$
\ind[D_\APS] 
=
-\ind[D_\aAPS]
= 
\dim \ker [D |_{C^\infty_\APS}] - \dim \ker [D |_{C^\infty_\aAPS}].
$$
\end{theorem}

\begin{proof}
It was shown in the proof of Theorem~\ref{thm:IndQ=IndD} that 
$$
\ker[D_\APS]=\ker[L_0]=\{u\in\FE^0(M;D) \mid Du=0 \mbox{ and } u|_{\Sigma_{t_1}} \in \ker[\Qmm]\}\, .
$$
Since $u|_{\Sigma_{t_1}}$ is smooth by Corollary~\ref{cor:Qglatt}, the solution $u$ of the Dirac equation lies in $\FE^s(M;D)$ for any $s\in\R$ by Theorem~\ref{thm:InhomoCauchyProblem}.
The Dirac equation
\begin{equation}
\frac{\partial u}{\partial t} = i N A_tu + \frac{n}{2}N Hu
\label{eq:Dirac}
\end{equation}
implies that $u$ is a $C^1$-section of $\{H^{s-1}(\Sigma_\bullet)\}$.
Differentiating \eqref{eq:Dirac} and iterating the argument shows that $u$ is a $C^k$-section of $\{H^{s-k}(\Sigma_\bullet)\}$ for any $k\in\N$ and any $s\in\R$.
By the Sobolev embedding theorem $u$ is smooth.
The argument for the kernel of $D_\aAPS$ is the same using $\Qpp$ instead of $\Qmm$.

Therefore we have $\ker[D_\APS]=\ker [D|_{C^\infty_\APS}]$ and, by Lemma~\ref{lem:kerQ},
\begin{align*}
\dim\mathrm{coker}[D_\APS]
=
\dim\ker[\Qmm^*]
=
\dim\ker[\Qpp]
=
\dim\ker[D|_{C^\infty_\aAPS}]\, .
\end{align*}
The argument for $\ind[D_\aAPS]$ is analogous.
\end{proof}

\section{The geometric index formula}
\label{sec:Geom}

In this section we give a geometric expression for the index considered in the previous section.
We use the splitting $M=[t_1,t_2]\times\Sigma$ and the Lorentzian metric $-N^2dt^2+g_t$ to introduce the auxiliary ``Wick rotated'' Riemannian metric $\gc := +N^2dt^2 + g_t$.
The formula analogous to \eqref{eq:DiracSurface+} for the \emph{Riemannian} Dirac operator $\Dc$ reads
\begin{equation}
\Dc u = -\betac \left(-\nablac_{\nu} + A_t + \frac{n}{2}\Hc \right) u .
\label{eq:DiracSurfaceR}
\end{equation}
Since the induced metrics $g_t$ on the slices $\Sigma_t$ have not changed, the operators $A_t$ in \eqref{eq:DiracSurface+} and in \eqref{eq:DiracSurfaceR} coincide.
Now we deform $\gc$ near the boundary $\partial M$ to a metric $\gd$ so that $\gd = dt^2 + g_{t_1}$ near $\{t_1\}\times\Sigma$ and $\gd = dt^2 + g_{t_2}$ near $\{t_2\}\times\Sigma$.
Then the corresponding Dirac operator $\Dd=-\betad(-\nablad_\nu+\Bd_t):C^\infty(M;S^+M)\to C^\infty(M;S^-M)$ simplifies to 
$$
\Dd =  -\betad \left( \frac{\partial}{\partial t} + A_{t_j} \right) 
$$
near $\{t_j\}\times\Sigma$ where $j=1,2$.

For any subset $I\subset\R$ denote by $\chi_I:\R\to\R$ the characteristic function of $I$.
Denote the corresponding spectral projections of $A_t$ by $P_I(t)=\chi_I(A_t)$ and the corresponding subspaces of $L^2(\Sigma_t;S^+M|_{\Sigma_t})$ by $L^2_I(\Sigma_t;S^+M|_{\Sigma_t}):=P_I(t)(L^2(\Sigma_t;S^+M|_{\Sigma_t}))$.

The operator $\Dd:H^1_\APS(M;S^+M) \to L^2(M;S^-M)$ is Fredholm where
\begin{align}
H^1_\APS(M;&S^+M) 
= 
\{u\in H^1(M;S^+M) \mid P_{[0,\infty)}(t_1)(u|_{\Sigma_{t_1}}) = 0 = P_{(-\infty,0]}(t_2)(u|_{\Sigma_{t_2}})\} \, .
\label{eq:defHAPS}
\end{align}
By the Atiyah-Patodi-Singer index theorem \cite[Thm.~3.10]{MR0397797} the index is given by
\begin{align}
\ind[\Dd]
&= 
\int_M \Adach (\nablad) - \frac{h(A_{t_1})+\eta(A_{t_1})+h(-A_{t_2})+\eta(-A_{t_2})}{2} \notag\\
&= 
\int_M \Adach (\nablad) - \frac{h(A_{t_1})+h(A_{t_2})+\eta(A_{t_1})-\eta(A_{t_2})}{2}
\label{eq:APS1}
\end{align}
where $\Adach (\nablad)$ is the $\Adach$-form computed from the curvature of the Levi-Civita connection $\nablad$ of $\gd$ and where $h(A)$ denotes the dimension of the kernel of $A$ and $\eta(A)$ its $\eta$-invariant.

Next we express the $\Adach$-integral in terms of the Levi-Civita connection $\nabla$ of the original Lorentzian metric $g$ rather than that of $\gd$.
Let $\adach$ be the $\mathrm{Ad}_{\mathrm{GL}(n+1,\R)}$-invariant polynomial giving rise to the $\Adach$-form,
$$
\Adach (\nabla) = \adach(\Omega^{\nabla},\ldots,\Omega^{\nabla}).
$$
Here $\Omega^{\nabla}$ denotes the curvature $2$-form matrix of $\nabla$ with respect to any local frame of $TM$.
Denote the connection $1$-form matrix of $\nabla$ by $\omega^{\nabla}$.
We define the \emph{trangression form} by
\begin{equation}
\TAdach(\nablad,\nabla) := \frac{n+1}{2} \int_0^1 \adach(\omega^{\nabla}-\omega^{\nablad},\Omega^{(1-s)\nabla+s\nablad},\ldots,\Omega^{(1-s)\nablad+s\nabla}) \, ds \, .
\label{eq:defTransgressionform}
\end{equation}
It is well known that 
$$
\Adach(\nablad) - \Adach(\nabla) = d\,\TAdach (\nablad,\nabla)\, .
$$
The Stokes theorem yields
\begin{equation}
\int_M \Adach (\nablad) = \int_M \Adach (\nabla) + \int_{\partial M}\TAdach (\nablad,\nabla) \, .
\label{eq:ADachStokes}
\end{equation}
Equation \eqref{eq:defTransgressionform} implies easily that the pullback of $\TAdach (\nablad,\nabla)$ to $\partial M$ is a polynomial expression in the curvature tensor of $\nabla$, the second fundamental form of the boundary with respect to $\nabla$, its first derivatives, and the first and second derivatives of the lapse function $N$.
Thus $\TAdach (\nablad,\nabla)$ is determined by the geometry of $(M,g)$ and does not depend on the particular choice of auxiliary metric $\gd$.
Henceforth we write $\TAdach (g)$ instead of $\TAdach (\nablad,\nabla)$.
Inserting \eqref{eq:ADachStokes} into \eqref{eq:APS1} yields
\begin{equation}
\ind[\Dd]
=
\int_M \Adach (\nabla) + \int_{\partial M}\TAdach(g) - \frac{h(A_{t_1})+h(A_{t_2})+\eta(A_{t_1})-\eta(A_{t_2})}{2} \, .
\label{eq:APS2}
\end{equation}

\subsection{Spectral flow}

We put $\xi(A):=\frac12 (h(A)+\eta(A))$.
Then the \emph{spectral flow} of the operator family $(A_t)_{t_1\le t\le t_2}$ is defined as
$$
\SF(A_{t\in [t_1,t_2]}) = j(t_2) 
$$
where $\xi(A_t)=c(t)+j(t)$ is the splitting into a continuous function $c(t)$ and an integer-valued function $j(t)$ with $j(t_1)=0$, cf.\ \cite[Sec.~7]{MR0397799}.
We rewrite \eqref{eq:APS2} as 
\begin{align}
\ind[\Dd]
&=
\int_M \Adach (\nabla) + \int_{\partial M}\TAdach(g) 
+\xi(A_{t_2})-\xi(A_{t_1})-h(A_{t_2}) \notag\\
&=
\int_M \Adach (\nabla) + \int_{\partial M}\TAdach(g) 
+c(t_2)+j(t_2)-c(t_1)-h(A_{t_2}) \notag
\end{align}
We rearrange the terms as
\begin{equation}
\ind[\Dd] - j(t_2) + h(A_{t_2}) = \int_M \Adach (\nabla) + \int_{\partial M}\TAdach(g) + c(t_2)-c(t_1) \, .
\label{eq:APS3}
\end{equation}
All terms on the left hand side of \eqref{eq:APS3} are integer valued.
On the other hand, if we consider the right hand side 
$$
\int_{[t_1,t_2]\times\Sigma} \Adach (\nabla) + \int_{\Sigma_{t_2}}\TAdach(g) - \int_{\Sigma_{t_1}}\TAdach(g) + c(t_2)-c(t_1)
$$
as a function of $t_2$ while keeping $t_1$ fixed, then all terms are continuous in $t_2$.
Hence both sides in \eqref{eq:APS3} are constant.
In the limit $t_2\to t_1$ the RHS tends to $0$.
Therefore
\begin{equation}
\ind[\Dd] - \SF(A_{t\in [t_1,t_2]}) + h(A_{t_2}) = \int_M \Adach (\nabla) + \int_{\partial M}\TAdach(g) + c(t_2)-c(t_1) = 0 \, .
\label{eq:APS4}
\end{equation}
We conclude
\begin{align}
\SF(A_{t\in [t_1,t_2]})
&=
\ind[\Dd] + h(A_{t_2})  \notag\\
&=
\int_M \Adach (\nabla) + \int_{\partial M}\TAdach(g) - \frac{h(A_{t_1})-h(A_{t_2})+\eta(A_{t_1})-\eta(A_{t_2})}{2}\notag\\
&=
\int_M \Adach (\nabla) + \int_{\partial M}\TAdach(g) + \xi(A_{t_2}) - \xi(A_{t_1}) \, .
\label{eq:APS5}
\end{align}

\subsection{Fredholm index of the wave propagator}

Next we relate $\ind[\Qmm]$ to the spectral flow $\SF(A_{t\in[t_1,t_2]})$.
We fix $t_1$ and consider $\ind[Q_{--}(t,t_1)]$ as a function of $t$.
For $t=t_1$, the operator $Q_{--}(t_1,t_1)$ is the embedding 
$$
L^2_{(-\infty,0)}(\Sigma_{t_1};S^+M|_{\Sigma_{t_1}})
\hookrightarrow 
L^2_{(-\infty,0]}(\Sigma_{t_1};S^+M|_{\Sigma_{t_1}})
$$
and hence has 
\begin{equation}
\ind[Q_{--}(t_1,t_1)]=-h(A_{t_1}) .
\label{eq:ind1}
\end{equation}
We use the characterization of spectral flow described in \cite{MR1426691}.
We choose a partition $t_1=\tau_0<\tau_1<\cdots<\tau_N=t_2$ and numbers $a_j>0$ such that $\pm a_j\notin \spec(A_t)$ for all $t\in[\tau_{j-1},\tau_j]$.
Then
\begin{equation}
\SF(A_{t\in[t_1,t_2]}) 
= 
\sum_{j=1}^N \left(\dim L^2_{[0,a_j]}(\Sigma_{\tau_j};S^+M|_{\Sigma_{\tau_{j}}}) - \dim L^2_{[0,a_j]}(\Sigma_{\tau_{j-1}};S^+M|_{\Sigma_{\tau_{j-1}}})\right) \, .
\label{eq:sf}
\end{equation}
Since $a_j\notin\spec(A_t)$ the family of projectors $P_{(-\infty,a_j]}(t)$ is a continuous section of bounded operators on the Hilbert space bundle $\{L^2(\Sigma_t;S^+M|_{\Sigma_t})\}_{t}$ over $[\tau_{j-1},\tau_j]$.
Thus 
$$
P_{(-\infty,a_j]}(t)\circ Q(t,t_1):L^2_{(-\infty,0)}(\Sigma_{t_1};S^+M|_{\Sigma_{t_1}})\to L^2_{(-\infty,a_j]}(\Sigma_{t};S^+M|_{\Sigma_{t}})
$$
is a strongly continuous family of Fredholm operators for $t\in[\tau_{j-1},\tau_j]$.
In particular, 
\begin{equation}
\ind[P_{(-\infty,a_j]}(\tau_j)\circ Q(\tau_j,t_1)]
=
\ind[P_{(-\infty,a_j]}(\tau_{j-1})\circ Q(\tau_{j-1},t_1)],
\label{eq:ind2}
\end{equation}
see \cite[Thm.~19.1.10]{Hormander2007} for details.
If we consider both operators $P_{(-\infty,a_j]}(t)\circ Q(t,t_1))$ and $P_{(-\infty,0]}(t)\circ Q(t,t_1))$ as operators $L^2_{(-\infty,0)}(\Sigma_{t_1};S^+M|_{\Sigma_{t_1}})\to L^2_{(-\infty,a_j]}(\Sigma_{t};S^+M|_{\Sigma_{t}})$, then they differ by $P_{(0,a_j]}(t)\circ Q(t,t_1))$.
Since the spectrum of $A_\tau$ is discrete, this difference operator is of finite rank and hence compact.
Therefore 
\begin{equation}
\ind[P_{(-\infty,a_j]}(t)\circ Q(t,t_1))] 
= 
\ind[P_{(-\infty,0]}(t)\circ Q(t,t_1))]
\label{eq:ind3}
\end{equation}
where both operators are considered as operators from $L^2_{(-\infty,0)}(\Sigma_{t_1};S^+M|_{\Sigma_{t_1}})$ to $L^2_{(-\infty,a_j]}(\Sigma_{t};S^+M|_{\Sigma_{t}})$.
Now $Q_{--}(t,t_1)$ is the same as $P_{(-\infty,0]}(t)\circ Q(t,t_1))$, except that it is considered as an operator to $L^2_{(-\infty,0]}(\Sigma_{t};S^+M|_{\Sigma_{t}})$.
Hence
\begin{equation}
\ind[Q_{--}(t,t_1)] 
= 
\ind[P_{(-\infty,0]}(t)\circ Q(t,t_1))] + \dim L^2_{(0,a_j]}(\Sigma_{t};S^+M|_{\Sigma_{t}})
\label{eq:ind4}
\end{equation}
for $t\in[\tau_{j-1},\tau_j]$.
We conclude
\begin{align*}
\ind[\Qmm]
&\stackrel{\phantom{\eqref{eq:ind4}}}{=}
\sum_{j=1}^N \left(\ind[Q_{--}(\tau_j,t_1)]-\ind[Q_{--}(\tau_{j-1},t_1)]\right) + \ind[Q_{--}(t_1,t_1)] \\
&\stackrel{\eqref{eq:ind4}}{=}
\sum_{j=1}^N \Big(\ind[P_{(-\infty,0]}(\tau_j)\circ Q(\tau_j,t_1)]+\dim L^2_{(0,a_j]}(\Sigma_{\tau_j};S^+M|_{\Sigma_{\tau_j}}) \\
&\;\;\;\;\;\;\;\;\;\;
-\ind[P_{(-\infty,0]}(\tau_{j-1})\circ Q(\tau_{j-1},t_1)] -\dim L^2_{(0,a_j]}(\Sigma_{\tau_{j-1}};S^+M|_{\Sigma_{\tau_{j-1}}})\Big) \\ 
&\;\;\;\;\;\;\;\;\;\;
+ \ind[Q_{--}(t_1,t_1)] \\
&\stackrel{\phantom{\eqref{eq:ind4}}}{=}
\sum_{j=1}^N \Big(\ind[P_{(-\infty,0]}(\tau_j)\circ Q(\tau_j,t_1)]+\dim L^2_{[0,a_j]}(\Sigma_{\tau_j};S^+M|_{\Sigma_{\tau_j}}) \\
&\;\;\;\;\;\;\;\;\;\;
-\ind[P_{(-\infty,0]}(\tau_{j-1})\circ Q(\tau_{j-1},t_1)] -\dim L^2_{[0,a_j]}(\Sigma_{\tau_{j-1}};S^+M|_{\Sigma_{\tau_{j-1}}}) \\ 
&\;\;\;\;\;\;\;\;\;\;
-h(A_{\tau_j}) + h(A_{\tau_{j-1}}) \Big) + \ind[Q_{--}(t_1,t_1)] \\
&\stackrel{\eqref{eq:sf}}{=}
\sum_{j=1}^N \Big(\ind[P_{(-\infty,0]}(\tau_j)\circ Q(\tau_j,t_1)]-\ind[P_{(-\infty,0]}(\tau_{j-1})\circ Q(\tau_{j-1},t_1)]\Big) \\
&\;\;\;\;\;\;\;\;\;\;
+ \SF(A_{t\in[t_1,t_2]}) - h(A_{t_2}) + h(A_{t_1}) + \ind[Q_{--}(t_1,t_1)] \\
&\stackrel{\eqref{eq:ind1}}{=}
\sum_{j=1}^N \Big(\ind[P_{(-\infty,0]}(\tau_j)\circ Q(\tau_j,t_1)]-\ind[P_{(-\infty,0]}(\tau_{j-1})\circ Q(\tau_{j-1},t_1)]\Big) \\
&\;\;\;\;\;\;\;\;\;\;
+ \SF(A_{t\in[t_1,t_2]}) -h(A_{t_2}) \\
&\stackrel{\eqref{eq:ind3}}{=}
\sum_{j=1}^N \Big(\ind[P_{(-\infty,a_j]}(\tau_j)\circ Q(\tau_j,t_1)]-\ind[P_{(-\infty,a_j]}(\tau_{j-1})\circ Q(\tau_{j-1},t_1)]\Big) \\
&\;\;\;\;\;\;\;\;\;\;
+ \SF(A_{t\in[t_1,t_2]}) -h(A_{t_2}) \\
&\stackrel{\eqref{eq:ind2}}{=}
\SF(A_{t\in[t_1,t_2]}) -h(A_{t_2}) \\
&\stackrel{\eqref{eq:APS5}}{=}
\int_M \Adach (\nabla) + \int_{\partial M}\TAdach(g) - \frac{h(A_{t_1})+h(A_{t_2})+\eta(A_{t_1})-\eta(A_{t_2})}{2}\, . \\
\end{align*}

We have proved:

\reqnomode
\begin{theorem}\label{thm:IndGeom}
The index of $\Qmm:L^2_{(-\infty,0)}(\Sigma_{t_1};S^+M|_{\Sigma_{t_1}})\to L^2_{(-\infty,0]}(\Sigma_{t_2};S^+M|_{\Sigma_{t_2}})$ is given by
\begin{align*}
\ind[\Qmm] 
= 
\int_M \Adach (\nabla) + \int_{\partial M}\TAdach(g) - \frac{h(A_{t_1})+h(A_{t_2})+\eta(A_{t_1})-\eta(A_{t_2})}{2}\, .
\tag*{\qed}
\end{align*}
\end{theorem}
\leqnomode

\begin{rem}
The spectral flow of the operator family $A_{t\in[t_1,t_2]}$ was crucial to connect the original Lorentzian index problem to a Riemannian one which we understand by the classical APS-index theorem.
On the Lorentzian side this relates directly to the wave propagator $Q(t_2,t_1)$ and its $Q_{--}(t_2,t_1)$-part.
This is based on the well-posedness of the Cauchy problem and there is no analog for $Q_{--}(t_2,t_1)$ in the Riemannian case.
\end{rem}


\section{An example}
\label{sec:example}

We describe an example with nontrivial index in any dimension divisible by $4$.
Let $k\in\N$.
The manifold will be $M=[t_1,t_2]\times S^{4k-1}$ equipped with a metric of the form $-dt^2 + g_t$ where $g_t$ is a $1$-parameter family of metrics on the sphere $S^{4k-1}$ arising as follows:

Consider the Hopf fibration $S^{4k-1} \to \C\P^{2k-1}$.
If $S^{4k-1}$ carries its canonical metric of constant sectional curvature $1$ and $\C\P^{2k-1}$ the Fubini-Study metric with sectional curvature between $1$ and $4$, then the Hopf fibration is a Riemannian submersion.
The fibers are great circles.
We now rescale the standard metric of $S^{4k-1}$ by the factor $t>0$ along the fibers and keep the metric unchanged on the orthogonal complement to the fibers.
This yields a $1$-parameter family of metrics $g_t$ on $S^{4k-1}$ known as \emph{Berger metrics}.

All Berger metrics are homogeneous under the unitary group $\mathrm{U}(2k)$.
Representation theoretic methods have been used to compute the full spectrum of the Dirac operator $A_t$ on $(S^{4k-1},g_t)$ in \cite[Prop.~3.2]{zbMATH03446111} for the case $k=1$ and in \cite[Thm.~3.1]{zbMATH00966366} for general $k$.
As a corollary one obtains a Dirac eigenvalue $\lambda(t) = (-1)^{k-1}(\tfrac{t}{2}-2k)$ of multiplicity $2k \choose k$.
This eigenvalue vanishes for $t=4k$ while all other Dirac eigenvalues on $(S^{4k-1},g_{4k})$ are nonzero.
Hence, if we choose $t_1$ slightly smaller than $4k$ and $t_2$ slightly bigger, then the spectral flow is $\SF(A_{t\in [t_1,t_2]}) = (-1)^{k-1} {2k \choose k}$.
We conclude
\begin{align*}
\ind[D_{\APS}] 
=
\ind[\Qmm]
&=
\SF(A_{t\in [t_1,t_2]}) 
= 
(-1)^{k-1} {2k \choose k}.
\end{align*}


\section{A parametrix for the APS Dirac operator in case the metric near the boundary is of product type}
\label{sec:Feynman}

Since the Dirac operator subject to APS boundary conditions is a Fredholm operator its equivalence class in the Calkin algebra is invertible.
In the special case when the Lorentzian metric is of product type near the boundary components $\Sigma_\pm$ it is in fact possible to construct an inverse modulo smoothing operators.

\ausblenden{Bild}{
\begin{pspicture}(-7,-3)(6,2.5)
\psset{viewpoint=-30 10 15, Decran=50, lightsrc=-20 20 15}
\defFunction{Lorentz}(u,v)
 {u Cos 0.4 mul 0.6 add v Cos mul}
 {u}
 {u Cos 0.4 mul 0.6 add v Sin mul}
\defFunction{HalsDick}(u,v)
 {v Cos}
 {u}
 {v Sin}
\defFunction{HalsDuenn}(u,v)
 {0.2 v Cos mul}
 {u}
 {0.2 v Sin mul}
\defFunction{Splus}(v)
 {v Cos}
 {-1}
 {v Sin}
\defFunction{Sminus}(v)
 {0.2 v Cos mul}
 {pi 0.5 add}
 {0.2 v Sin mul}

\psSolid[object=surfaceparametree,
        base=-1 0 0 2 pi mul,
        fillcolor=RoyalBlue!70,
        incolor=white,
        opacity=0.7,
        function=HalsDick,
        ngrid=20 180,
        grid=false]%
\psSolid[object=surfaceparametree,
        base=0 pi 0 2 pi mul,
        fillcolor=blue!70,
        incolor=black,
        opacity=0.7,
        function=Lorentz,
        ngrid=20 180,
        grid=false]%
\psSolid[object=surfaceparametree,
        base=pi pi 0.5 add 0 2 pi mul,
        fillcolor=RoyalBlue!70,
        incolor=yellow!50,
        opacity=0.7,
        function=HalsDuenn,
        ngrid=120 180,
        grid=false]%
\psSolid[object=courbe,
        range=1.1 4.22,
        r=0,
        ngrid=360,
        linecolor=black,
        linewidth=0.02,
        function=Splus]%
\psSolid[object=courbe,
        range=0 2 pi mul,
        r=0,
        ngrid=360,
        linecolor=black,
        linewidth=0.02,
        function=Sminus]%
       
\psPoint(-3,0,1){M}
\uput[u](M){\psframebox*[framearc=.3]{$M$}}
\psPoint(-3,3.7,1.5){S-}
\uput[u](S-){\psframebox*[framearc=.3]{$\Sigma_-$}}
\psPoint(-3,-1.5,1){S+}
\uput[u](S+){\psframebox*[framearc=.3]{$\Sigma_+$}}
\end{pspicture}
\begin{center}
\textbf{Fig.~2.}
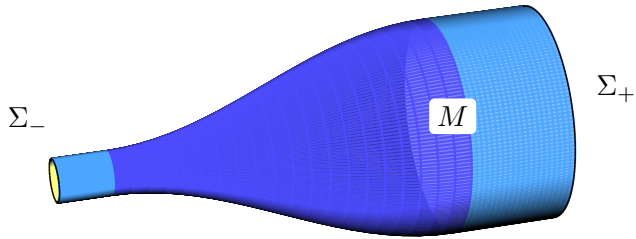
\emph{The manifold $M$ with product structure near the boundary}
\end{center}
}

Such a parametrix has the same microlocal properties as the Feynman propagator and we will refer to it as a Feynman parametrix.
The Feynman propagator is a fundamental solution of the Dirac operator that was introduced by Richard Feynman to describe the physics of the electron in Quantum Electrodynamics. Feynman parametrices appear in quantum field theory on curved spacetimes and were constructed as one type of distinguished parametrix in the seminal paper of Duistermaat and H\"ormander \cite{MR0388464}. 

\subsection{Feynman parametrices}

Let $L$ be a normally hyperbolic operator on $X$ acting on the sections of a vector bundle $E$, see Appendix~\ref{app:NormHyp} for details.
Let $G$ be a left parametrix for $L$, i.e.\ $G$ is a continuous map $G: C^\infty_0(X;E) \to C^\infty(E)$ and
$$
 G L = \mathrm{id}_{C^\infty_0} + R,
$$
where $R$ is an operator with smooth integral kernel.
As usual, define $\mathrm{WF}'(G)$ as the set
$$
 \mathrm{WF}'(G) = \{(x,\xi;x',\xi') \in \dot T^*(X \times X) \mid (x,\xi;x',-\xi') \in \mathrm{WF}(G)\},
$$
where $\dot T^*(X \times X)=T^*(X \times X) \backslash 0$, and $\mathrm{WF}(G)$ is the wavefront set of the
distributional integral kernel of $G$.
Let $\Delta^*$ be the diagonal in $\dot T^*(X \times X)$ and $\Phi_t$ the geodesic flow on $\dot T^*X$.
We put
$$
 \Lambda_F:=\Delta^* \cup \{(x,\xi;x',\xi') \in \dot T^*(X \times X) \mid \xi \textrm{ is lightlike and}\;  
 \exists\, t>0: \Phi_t(x',\xi') = (x,\xi) \}.
$$
Then we say that $G$ is a \emph{Feynman parametrix} if
$$
 \mathrm{WF}'(G) \subset \Lambda_F .
$$

Feynman parametrices are known to exist and are unique up to smoothing operators on any globally hyperbolic spacetime
(see \cite[Thm.~6.5.3]{MR0388464} taking into account that globally hyperbolic spacetimes are pseudoconvex with respect to normally hyperbolic operators).
Every Feynman parametrix is also a right parametrix, i.e.\
$$
 L G = \mathrm{id}_{C^\infty_0} + \tilde R ,
$$
and extends to a continuous map $G: H^s_0(X;E) \to H^{s+1}_{\loc}(X;E)$ for all $s \in \R$.

Analogously, a left parametrix $S$ for a Dirac type Dirac operator $\D$ called a Feynman parametrix if $\mathrm{WF}'(S) \subset \Lambda_F$. 
By the same reasoning as in \cite[Thm.~6.5.3]{MR0388464} such parametrices are unique up to smoothing operators. 
Since $L=\D^2$ is normally hyperbolic there exists a Feynman parametrix  $G$ for $L$. The operator $\D^2$ is a direct sum of $D \tilde D$ and $\tilde D D$.
The Feynman parametrix  $G$ can therefore also be assumed to be the direct sum of Feynman parametrices for $D \tilde D$ and $\tilde D D$, respectively.
\begin{lem}
 The right parametrix $\mathcal{S}:=\D G$ for $\D$ is a Feynman parametrix.
 It maps $H^s_0(X;S^\pm X)$ continuously to $H^{s}_{\loc}(X;S^\mp X)$ for all $s \in \R$.
\end{lem}
\begin{proof}
 The wavefront set relation is preserved by the composition with $\D$ from the left.
 The only thing to show here is that $\mathcal{S}$ is a left parametrix.
 This follows from a variation of the argument for the uniqueness of distinguished parametrices in \cite{MR0388464}.
 Namely, let $\mathcal{S}'$ be the left parametrix $G \D$, and define $R := \mathcal{S} - \mathcal{S}'$.
 We will show that $R$ is smoothing, i.e.\ $\mathrm{WF}'(R)=\emptyset$.
 
 Clearly, $\mathrm{WF}'(R)\subset\Lambda_F$, so let $(x,\xi;x',\xi') \in \Lambda_F$.
 This implies that either $(x,\xi)=(x',\xi')$ or there exists a lightlike geodesic connecting $(x,\xi)$ and $(x',\xi')$.
 Denote this geodesic arc by $K\subset X$ and choose $\vp \in C^\infty_0(X)$ such that $\vp(x)=1$ for all $x$ in an open neighborhood of $K$.
 Then, since the first order operator $\vp \D - \D \vp$ is supported away from $K$, the point $(x,\xi;x',\xi')$ is not in the wavefront set of $\vp \D - \D \vp$.
 Since
 $$
  \mathcal{S} (\vp \D - \D \vp) \mathcal{S}' = \mathcal{S} \vp - \vp \mathcal{S}'
 $$
 the point $(x,\xi;x',\xi')$ is not in the wavefront set of $\mathcal{S} \vp - \vp \mathcal{S}'$.
 Since this is true for any $(x,\xi;x',\xi') \in \Lambda_F$ the operator $\mathcal{S} - \mathcal{S}'$ has smooth kernel.
 \end{proof}

\subsection{A Feynman parametrix in the product case}
It is instructive to see what a Feynman parametrix for the normally hyperbolic operator $L= \D^2$ looks like in the case $X = \R \times \Sigma$ equipped with a product metric. 
In this case $L$ has the form
$$
 L=\frac{\partial^2}{\partial t^2} + A^2.
$$
The operator $|A|^{-1}$ as given by spectral calculus  has domain 
$\ker[A]^\perp$ so that the operator $|A|^{-1} (1-P_0)$ is well defined and bounded on $L^2(\Sigma;S M)$. 
Here $P_0=P_{\{0\}}$ is the orthogonal projection onto the kernel of $A$. 
Then the family of operators
$$
 g(t) := \frac{1}{2 \rmi}  \left( e^{\rmi |A| t} \chi_{(0,\infty)}(t) +
 e^{-\rmi |A| t}  \chi_{(-\infty,0)}(t) \right)  |A|^{-1}  (1-P_0) + t \, \chi_{(0,\infty)}(t) P_0
$$
defines a Feynman parametrix via
$$
 (G u)(t,\cdot) = \int_\R g(t-s) u(s,\cdot) ds .
$$
For any open interval $I$ the restriction of $G$ to $I \times \Sigma$ then defines a Feynman parametrix
on the globally hyperbolic manifold $I \times \Sigma$.

Assume $s>1/2$.
For $M \subset X$ of the form $[t_1,t_2] \times \Sigma$ we have for any $u \in H^s(M;SM)$ the following formulae, easily obtained by integration by parts:
\begin{gather} \label{remainderformula}
 \left( \frac{\partial}{\partial t} + \mathrm{i}\; A \right) G \left( \frac{\partial}{\partial t} - \mathrm{i}\; A \right) u
 = u + R u,
\end{gather}
and
$$
  \left(G \left( \frac{\partial}{\partial t} \pm \mathrm{i}\; A \right) -  
 \left( \frac{\partial}{\partial t} \pm \mathrm{i}\; A \right) G \right) u =
 Q  u.
$$
Here the map
$R: H^s(M;SM) \to H^{s-1/2}(M;SM)$ is defined as
$$
 (R u)(t,\cdot):=  e^{\rmi |A| (t-t_1)} P_{(0,\infty)}(t_1) (u|_{\Sigma_{t_1}}) + e^{-\rmi |A| (t-t_2)} P_{(-\infty,0)}(t_2)  (u|_{\Sigma_{t_2}}) - P_0 (u|_{\Sigma_{t_1}})
$$
The operator $Q: H^s(M;SM) \to H^{s+1/2}(M;SM)$ is given by
$$
 (Q u)(t,\cdot):=\frac{1}{2 \rmi} |A|^{-1} (1-P_0)\left( e^{\rmi |A|(t-t_1)} (u|_{\Sigma_{t_1}}) - 
 e^{-\rmi |A| (t-t_2)} (u|_{\Sigma_{t_2}}) \right) + (t-t_1) P_0 (u|_{\Sigma_{t_1}}).
$$
We define the \emph{Feynman propagator} as $\mathcal{S}:=  \D G$.
One computes directly for $u \in H^s(M;S^-M)$
$$
 (\mathcal{S} u)(t,\cdot) = \left( \left( \frac{\partial}{\partial t} + \mathrm{i}\; A \right) G \beta u \right)(t,\cdot) = \int_\R h(t-s) \beta u(s,\cdot) ds\;,
$$
where
$$
 h(t) =  P_{(0,\infty)}(t) e^{\rmi |A| t} \chi_{(0,\infty)}(t) - P_{(-\infty,0)}(t) e^{-\rmi |A| t} \chi_{(-\infty,0)}(t)   +  \chi_{(0,\infty)}(t) P_0\;.
$$
Thus the Feynman propagator propagates elements in the positive spectral subspace of $A$ forward in time and elements in the negative 
spectral subspace of $A$ backwards in time.

\begin{theorem} \label{feynprop1}
Let $M=[t_1,t_2] \times \Sigma$ be equipped with a product metric $-dt^2+g$ and let $\mathcal{S}=  \D G$ be the Feynman propagator.
Let $s>1/2$.
Then we have
\begin{align*}
 D \mathcal{S}|_{H^s(M;S^-M)} & =  \mathrm{id}_{H^s(M;S^-M)},\\
 \mathcal{S} D|_{H^s(M;S^+M)} & = \mathrm{id}_{H^s(M;S^+M)}+ R.
\end{align*}
Moreover, $\mathcal{S}$ maps $H^s(M;S^-M)$ to sections satisfying APS boundary conditions at $t_1$, i.e.\
if $u \in H^s(M;S^-M)$ then $P_{[0,\infty)}(t_1)((\mathcal{S}u)|_{\Sigma_{t_1}})= 0$.
We also have $$\mathcal{S} D|_{H^s_\APS(M;S^+M)} = \mathrm{id}_{H^s_\APS(M;S^+M)}.$$
\end{theorem}

\begin{proof}
That $\D \mathcal{S}|_{H^s(M;SM)} = \mathrm{id}_{H^s(M;SM)}$ can be checked by direct computation and 
$ \mathcal{S} \D|_{H^s(M;S^+M)}  = \mathrm{id}_{H^s(M;S^+M)}+ R$ follows immediately from \eqref{remainderformula}.
Since $h(t)=- P_{(-\infty,0)}(t) e^{-\rmi |A| t}$ for $t<0$ we obtain  $P_{[0,\infty)}(t_1)((\mathcal{S}u)|_{\Sigma_{t_1}})= 0$. If 
$u \in H^s(M;S^-M)$ then $\mathcal{S}u \in H^s(M;S^+M)$ and this means it satisfies APS-boundary conditions.
The formula $\mathcal{S} D|_{H^s_{\APS}(M;S^+M)} = \mathrm{id}_{H^s_{\APS}(M;S^+M)}$ follows from the fact that
the operator $R$  vanishes on $H^s_\APS(M;S^+M)$.
\end{proof}

\subsection{Gluing of Feynman parametrices} Let $(X,g)$ be a globally hyperbolic manifold
diffeomorphic to $(-T,T) \times \Sigma$ such that $\{ t\} \times \Sigma$
is a Cauchy hypersurface for each $t \in (-T,T)$. Let $0 < \delta < T$. The sets $U_1 = (-T,\delta) \times \Sigma$, $U_2=(-\delta,T) \times \Sigma$
and $U=U_1 \cap U_2 = (-\delta,\delta) \times \Sigma$. Let $L$ be a Dirac type operator or a normally hyperbolic operator.
Now suppose that $G_1$ is a Feynman parametrix on $U_1$ and $G_2$ a Feynman parametrix on $U_2$.
We choose the following gluing functions 
$\theta_1, \theta_2, \phi_1,\phi_2 \in C^\infty(X)$ satisfying the following properties
\begin{itemize}
 \item $\supp(\theta_1) \subset \supp(\phi_1) \subset U_1$,
 \item $\supp(\theta_2) \subset \supp(\phi_2) \subset  U_2$,
 \item $\theta_1 + \theta_2 = 1$,
 \item $\phi_1(x) = 1$ for all $x$ in some open neighborhood of $\supp(\theta_1)$,
 \item $\phi_2(x) = 1$ for all $x$ in some open neighborhood of $\supp(\theta_2)$.
\end{itemize}
This implies that the differential operator $[\phi_1,L]=\phi_1 L - L \phi_1$
is compactly supported in $X$, and for all $x \in \supp([L,\phi_1])$ we have $\theta_1(x)=0$,
$\theta_2(x)=1$, $\phi_2(x)=1$. The analogous statement holds for $[L,\phi_1]$.
\begin{lem}
 The operator
 $$
  G := \theta_1 G_1 \phi_1  +\theta_2 G_2 \phi_2 - \theta_1 G_1 [\phi_1, L] G_2 \phi_2 - \theta_2 G_2 [\phi_2, L] G_1 \phi_1.  
 $$
 is a Feynman parametrix.
\end{lem}
\begin{proof}
 It is easy to see that $G$ is well defined.
 We first check that $\mathrm{WF}'(G) \subset \Lambda_F$. 
 Since $G_1$ and $G_2$ are Feynman parametrices both $\mathrm{WF}'(\theta_1 G_1 \phi_1 )$
 and $\mathrm{WF}'(\theta_2 G_2 \phi_2)$ are subsets of $\Lambda_F$.
 Moreover, by the mapping properties of the canonical relation  $\Lambda_F$, we have
 $$
  \mathrm{WF}'(\theta_2 G_2 [\phi_2, L] G_1 \phi_1) \subset \{(x,\xi;x',\xi') \in \Lambda_F \mid \xi \in V_+, x \in U_2, x' \in U_1\},
 $$
 where $V_+ \subset T^*M$ is the closed forward light cone, i.e.\ the set of future-directed causal covectors.
 We have used here that the past of $\supp([\phi_2, L])$ has empty intersection with $\supp(\theta_2)$.
 Similarly, using that the future of $\supp([\phi_1, L])$ has empty intersection with $\supp(\theta_1)$, one obtains
 $$
  \mathrm{WF}'(\theta_1 G_1 [\phi_1, L] G_2 \phi_2) \subset \{(x,\xi;x',\xi') \in \Lambda_F \mid -\xi \in V_+, x \in U_1, x' \in U_2\}.
 $$
 This shows that $\mathrm{WF}'(G) \subset \Lambda_F$.
 Now we check that $G$ is a left parametrix. 
 Let us introduce the relation $A \sim B$ if $A-B$ is a smoothing operator.
 Note that
 $$
  \theta_1 G_1 \phi_1 L  \sim \theta_1 \phi_1 + \theta_1 G_1 [\phi_1,L]
 $$
 and
 $$
  \theta_1 G_1 [\phi_1, L] G_2 \phi_2 L \sim \theta_1 G_1 [\phi_1, L] G_2 [\phi_2, L] + \theta_1 G_1 [\phi_1,L].
 $$
 Collecting these two terms we obtain 
 $$
 \theta_1 G_1 \phi_1 L - \theta_1 G_1 [\phi_1, L] G_2 \phi_2 L \sim \theta_1 \phi_1 + R_1,
 $$
 where  $R_1 = -\theta_1 G_1 Q_1$, and $Q_1=[\phi_1, L] G_2 [\phi_2, L]$.
 Since $\mathrm{WF}'(G_2) \subset \Lambda_F$ we have 
 $$
  \mathrm{WF}'(Q_1) \subset \{(x,\xi;x',\xi') \in \Lambda_F \mid x \in \supp([\phi_1, L]), x' \in  \supp([\phi_2, L]), \xi \in V_+\}.
 $$
 Again, the future of $\supp([\phi_1, L])$ does not intersect $\supp(\theta_1)$ and hence
 we obtain $\mathrm{WF}'(R_1) = \emptyset$. Thus, $R_1$
 is a smoothing operator.
 In the same way one shows that
 $$
 \theta_2 G_2 \phi_2 L = \theta_2 \phi_2 + R_2,
 $$
 where $R_2$ is a smoothing operator.
 All together we obtain
 $$
  G L = \theta_1 \phi_1 + \theta_2 \phi_2 + R_1 + R_2 = \mathrm{id}_{C^\infty_0} + R, 
 $$
 where $R=R_1 + R_2$ is a smoothing operator, and hence $G$ is a left parametrix.
\end{proof}

\subsection{A distinguished Feynman parametrix if \boldmath$M$ has product structure near the boundary}

In this subsection we will be assuming that 
$M$ is isometric to a Lorentzian cylinder near the boundaries $\Sigma_{t_1}$ and $\Sigma_{t_2}$. This means near these boundaries
the metric is of the form
$-dt^2 + g$ where $g$ is a Riemannian metric on $\Sigma$ independent of $t$. More precisely, we assume
product structure on $(t_1,t_1 + \epsilon) \times \Sigma$ and $(t_2,t_2-\epsilon) \times \Sigma$. Thus,
we can think of $M$ as the union of the two product manifolds $[t_1,t_1 + \epsilon] \times \Sigma_1$,  
$[t_2,t_2-\epsilon] \times \Sigma_2$ and the open globally hyperbolic manifold $X=(t_1,t_2) \times \Sigma$.
Starting with any Feynman parametrix for the operator $\D$ on $X$ we can use the Feynman parametrices
on $[t_1,t_1 + \epsilon] \times \Sigma_1$, $[t_2,t_2 - \epsilon] \times \Sigma_2$ to obtain a Feynman parametrix
$\mathcal{S}$ by gluing. The following theorem is immediately deduced from the support properties of the gluing functions and Theorem~\ref{feynprop1}.

\begin{theorem}\label{modcomp1}
 Assume  that the metric on $M$ has product structure near $\Sigma_1$ and $\Sigma_2$. Let $s>1/2$.
 Then the operator $ \mathcal{S}$ continuously maps $H^s(M;S^- M))$ to $H^s_\APS(M;S^+ M)$ and we have
 \begin{gather*}
  D \mathcal{S}|_{H^s(M;S^-M)}  = \mathrm{id}_{H^s(M;S^-M)} + \mathcal{Q},\\
  \mathcal{S} D |_{H^s_\APS(M;S^+M)}  = \mathrm{id}_{H^s_\APS(M;S^+M)} + \mathcal{R}
 \end{gather*}
 where $\mathcal{Q}$ and $\mathcal{R}$ have integral kernels that are smooth up to $\partial M$. 
 Moreover, the integral kernel of $Q$ is compactly supported
 in $(M \backslash \partial M) \times M$, and the integral kernel of $\mathcal{R}$ satisfies APS boundary conditions in the first variable.
 Therefore, for  $\mathcal{Q}$ and $\mathcal{R}$ have the following mapping properties:
  \begin{align*}
   \mathcal{Q} : H^r(M; S^-M) &\to C_0^\infty(M \backslash \partial M;S^-M),\\
  \mathcal{R} : H^r(M;S^+M) &\to C^\infty_\APS(M;S^+M),
 \end{align*}
 for any $r \in \R$.\qed
\end{theorem}

The above can be understood as a refined version of the Fredholm property of the operator $D$ subject to APS-boundary conditions and it gives an alternative approach to the index problem via parametrices. 
An example is the following refinement of Lemma~\ref{FIOlemma} in the case of product structure of the metric near the boundary.

\begin{theorem}\label{smoothy1}
 Assume that the metric on $M$ has product structure near $\Sigma_{t_1}$ and $\Sigma_{t_2}$.
Then the operators $\Qpm \circ P_{(-\infty,0)}(t_1)$ and $\Qmp \circ P_{(0,\infty)}(t_1)$ have smooth integral kernels.
\end{theorem}
\begin{proof}
 In the proof we will assume that the Feynman parametrix was constructed by gluing on  a slightly larger spacetime
 $\tilde M = [t_1,t_3] \times \Sigma$, where $M = [t_1,t_2] \times \Sigma$ and we assume the metric to be of product type on 
 $ Z=[t_2,t_3] \times \Sigma$. We will also assume that the gluing functions used in the construction have derivatives supported away from $Z$.
 We will need a smooth function $\theta$  on $\tilde M$ which is equal to one on $M$ and which vanishes near $t=t_3$. We choose $\theta$ such that
 on $Z$ it depends on the $t$-variable only (i.e. such that $\theta$ is constant on $\Sigma_t$).
  Now let $u \in P_{(-\infty,0)}(t_1) H^s(\Sigma_{t_1},S^+M|_{\Sigma_{t_1}})$. Then there exists a unique solution $\Phi \in \FE^s(\tilde M,S^+\tilde M)$ of the Cauchy problem
 $D \Phi=0,\; \res_{t_1}(\Phi) =u$. 
 The spinor $D(\theta \Phi) = \beta (\partial_t \theta) \Phi$ is compactly supported in $Z$.
 Since $\theta\Phi$ satisfies APS-boundary conditions on $\tilde M$ we have
 $$
  \mathcal{S} ( \beta (\partial_t \theta) \Phi ) = \theta \Phi + K \Phi,
 $$
 where $K$ is an operator with smooth integral kernel.
 This shows that $Q(t_2,t_1)$ equals to $(\res_{t_2} \circ \mathcal{S}) ( \beta (\partial_t \theta) \Phi ) + (\res_{t_2} \circ K)\Phi$.
 On $Z$ we decompose $\Phi(t,\cdot) = \Phi_+(t,\cdot) + \Phi_-(t,\cdot)+\Phi_0$ according to the spectral decomposition into positive, negative, and zero spectral subspaces of $A$.
 Hence  $\Qpm(u) = (\res_{t_2} \circ \mathcal{S})( \beta (\partial_t \theta) (\Phi_- + \Phi_0) ) + (\res_{t_2} \circ K)\Phi$. Since $\beta$ anticommutes with $A$ and $\partial_t \theta$
 depends only on $t$, the section $(\beta (\partial_t \theta) \Phi_-)(t,\cdot)$ is in the positive spectral subspace of $A$ for each $t$. By the propagation
 properties of $\mathcal{S}$ and the support properties of the gluing functions we obtain $(\res_{t_2} \circ \mathcal{S}) ( \beta (\partial_t \theta) \Phi_- )=0$
 and therefore $\Qpm u = (\res_{t_2} \circ \mathcal{S})( \beta (\partial_t \theta) \Phi_0) + (\res_{t_2} \circ K)\Phi$. The map 
 $$
  u \mapsto (\res_{t_2} \circ \mathcal{S})( \beta (\partial_t \theta) \Phi_0) 
 $$
 has smooth integral kernel because the kernel of $A$ is finite dimensional and consists of smooth sections.
 A similar argument works for $\Qmp \circ P_{(0,\infty)}(t_1)$.
\end{proof}


\section{Concluding remarks}
\label{sec:remarks}

To keep notation simple we have restricted our presentation to the classical Dirac operator acting on spinor fields.
The same proof actually yields the index theorem in a more general situation.
First, we can replace the spin structure by a spin$^c$ structure.
Then $M$ carries the associated determinant bundle $L\to M$ which is a Hermitian line bundle.
We assume that it is equipped with a metric connection $\nabla^L$.
Then the spinor bundles $S^\pm M$ and the spin$^c$-Dirac operators are defined.

In addition, we can twist the spin$^c$-Dirac operator with a Hermitian vector bundle $E\to M$ also carrying a metric connection $\nabla^E$.
The resulting twisted spin$^c$-Dirac operator then acts on spinors with coefficients in $E$, i.e., on sections of $S^{(\pm)}M\otimes E$.

\begin{theorem}\label{thm:maingeneral}
Let $(M,g)$ be a compact time-oriented globally hyperbolic Lorentzian manifold with boundary $\partial M = \Sigma_- \sqcup \Sigma_+$. 
Here $\Sigma_\pm$ are smooth spacelike Cauchy hypersurfaces, with $\Sigma_+$ lying in the future of $\Sigma_-$.
Assume that $M$ is even dimensional and comes equipped with a spin$^c$ structure.
Let $\nabla^L$ be a metric connection on the determinant line bundle, let $E\to M$ be a Hermitian vector bundle equipped with a metric connection $\nabla^E$.

Then the twisted spin$^c$-Dirac operator $D_\APS: \FE^0_\APS(M;S^+M\otimes E) \to  L^2(M;S^-M\otimes E)$ under Atiyah-Patodi-Singer boundary conditions is Fredholm and its index is given by
$$
\ind[D_\APS]
= 
\int_M \Adach (\nabla)\wedge e^{c_1(\nabla^L)/2}\wedge \ch(\nabla^E) + \int_{\partial M}\mathcal{T} - \frac{h(A_{-})+h(A_{+})+\eta(A_{-})-\eta(A_{+})}{2}\, .
$$ 
\end{theorem}

Here, $A_\pm$ are the induced Riemannian twisted spin$^c$-Dirac operators on $\Sigma_\pm$, $c_1(\nabla^L)$ is the Chern form of the curvature of $\nabla^L$ and $\ch(\nabla^E)$ is the Chern character form of the curvature of $\nabla^E$.
The transgression form $\mathcal{T}$ is defined as in \eqref{eq:defTransgressionform} for the given connections on $SM\otimes E$ and a connection induced by a metric with product structure near the boundary and connections $\hat\nabla^L$ and $\hat\nabla^E$ also having product structure near the boundary.
In particular, the boundary integral vanishes if the given metric and connections $\nabla^L$ and $\nabla^E$ have product structure near the boundary.

This more general version of the main theorem covers the physically relevant cases of hermitian bundles induced by compact gauge groups. 
For example, the spin$^c$ case allows one to include electromagnetic potentials.
Examples for the index of twisted Dirac operators on Lorentzian manifolds can be found in \cite[Sec.~3]{baerstroh2015chiral}.

In the special case when the spacetime has a metric of product type $-dt^2 + g$, with $g$ independent of $t$, one can consider the case of a time-dependent
connection and potential and use a scattering theory approach to study solutions of the Dirac equation. In this case an index theorem for the
scattering operator for the Dirac equation on positive spectral subspaces  is known and worked out in detail in \cite{bunke1992index}. In the case of Minkowski 
spacetime such index formulae were also obtained earlier in \cite{matsui1987} and \cite{matsui1990index}. The formulae there involve a Chern-Simons class and the $\Adach$-form
of the Cauchy surface rather than the spacetime. The $\eta$-invariant does not appear in this case.

A word of caution here:
If the twist bundle $E$ is a natural bundle induced by the geometry of $M$ itself, e.g.\ $E=TM$, then its induced metric is not definite in general.
If this happens then Theorem~\ref{thm:maingeneral} does not apply.
Indeed, the twisted spin$^c$-Dirac operator may no longer be a symmetric hyperbolic system in this case, the induced operators on the boundary may no longer be selfadjoint for a Hermitian metric and the whole analysis will require a careful reexamination.

As in elliptic index theory there are many directions in which a generalization should be possible.
It is conceivable that one could work out an equivariant version, a family version, a spatially $L^2$-version or other versions of our index theorem for noncompact Cauchy hypersurfaces.
The latter may be very useful on physical grounds.

For the Dirac operator on a compact Riemannian spin manifold with boundary one cannot use anti-Atiyah-Patodi-Singer boundary conditions instead of APS-boundary conditions.
This would not give rise to a Fredholm operator.
In the Lorentzian setting however anti-APS conditions are equally good.
This difference in qualitative behavior can be understood as follows:
Solving the Dirac equation on a Riemannian product manifold under APS-boundary conditions is like solving a heat equation which is possible forward in time but not backwards in time.
In the Lorentzian setting it corresponds to solving a wave equation which can be done both forward and backwards in time.

As a final remark we would like to note that APS boundary conditions make sense also for distributional sections of lower regularity. 
For example, for the restriction
of a distributional spinor $u$ to $\Sigma_\pm$ to be well defined it is sufficient to assume that its Sobolev wavefront set $\mathrm{WF}_s(u)$
is contained in the light cone for some $s>1/2$. This is automatically satisfied if $u$ solves the Dirac equation. Distributional solutions that satisfy APS boundary conditions can be shown to be smooth, and therefore allowing lower regularity results in the same theory.


\appendix 
\section{Solution of the Cauchy problem for normally hyperbolic operators}
\label{app:NormHyp}

Let $L$ be a normally hyperbolic operator on a globally hyperbolic spacetime $X$ acting on the sections of the vector bundle $E$.
Then there exists a unique connection $\nabla: C^\infty(X; E) \to C^\infty(X; E \otimes T^*X)$ and a potential $V \in C^\infty(X;\mathrm{End}(E))$
such that 
$$
 L = \Delta^{\nabla} + V,
$$
where $\Delta^{\nabla} = -\mathrm{Tr}_g (\nabla^{E \otimes T^* X} \circ \nabla)$ is the connection Laplacian, see \cite[Prop.~3.1]{zbMATH00955339}. 
Here $\nabla^{E \otimes T^* X}$ is the connection on $E \otimes T^* X$ induced by the connection on $E$ and the Levi-Civita connection on $T^*X$.

Note that $L$ is of real principal type and the subprincipal symbol $\mathrm{sub}(L)$ of $L$ can be expressed in local bundle-charts in terms of the connection one
form $A$ of $\nabla$ (see \cite[Prop.~3.1]{jakobson2007high} or Example 4.2 in \cite{LI2016204} for a detailed explanation) as $\mathrm{sub}(L)(\xi) = - 2 \rmi g(\xi, A)$.

Let $\Sigma \subset X$ be a Cauchy hypersurface and let $\nu$ be its past directed unit-normal vector field. 
Then it is well known that the Cauchy problem for $L$ is well posed and can be solved in the sense that there exist continuous operators $\G: C_0^\infty(\Sigma;E) \to C^\infty(X;E)$
and $\T: C_0^\infty(\Sigma;E) \to C^\infty(X;E)$ such that for any $g,f \in C_0^\infty(\Sigma;E)$ the section
$u = \G f + \T g$ is the unique solution of the initial value problem
\begin{align*}
 L u = 0,  \quad 
 u |_\Sigma = f, \quad (-\nabla_\nu u) |_\Sigma = g.
\end{align*}

The solution operators $\G, \T$ are well known to be Fourier integral operators. 
Unfortunately we were able to find the statement of this fact only for the case of scalar operators. 
Therefore, for the convenience of the reader, we revise the theory in the case of operators acting on vector bundles.

For two pseudo-Riemannian manifolds $Y,Z$ and vector bundles $F \to Z$ and $E \to Y$ suppose that $C \subset T^* Y \times T^*Z$ is a homogeneous canonical relation.
Let $C' = \{  (y,\xi,z,\eta)  \mid (y,\xi,z,-\eta) \in C \}$.
Denote by $I^m(Y \times Z,C;\mathrm{Hom}(F,E))$ the set of Fourier integral operators with canonical relation $C$.
These are operators
whose distributional Schwartz kernels take values in the bundle $E \boxtimes F^*$ and
can be represented locally by an oscillatory integral with a polyhomogeneous matrix valued symbol and a scalar phase 
function that locally parametrizes $C'$. We refer to \cite{hormander1971fourier} and \cite{Hormander:1985aaiv} for precise definitions. 
On $C$ we have the Keller-Maslov line bundle $M_C$. The principal symbol $\sigma_U$
of a Fourier integral operator $U \in I^m(Y \times Z,C;\mathrm{Hom}(F,E))$ is a section in the bundle
$$
 j_{C}^* \left( E \boxtimes F^* \right) \otimes \Omega_{\frac{1}{2}} \otimes M_C,
$$
where $j_C: C \to Y \times Z$ is the inclusion composed with the projection, and $\Omega_{\frac{1}{2}}$ is the bundle of half-densities on $C$.
A choice of phase function that parametrizes the Lagrangian manifold $C'$ gives a local trivialization of the Keller-Maslov line bundle.

\begin{theorem} \label{fioth}
 Let $L$ as above a normally hyperbolic operator acting on sections of a vector bundle $E$ over a globally hyperbolic spacetime.
 Let $\Sigma$ be a Cauchy hypersurface and let $\G$ and $\T$ be the operators solving the Cauchy problem.
 Then $\G$ and $\T$ are Fourier integral operators. More precisely,
 \begin{align*}
  \G \in I^{-\frac{1}{4}}(X \times \Sigma; C; \mathrm{Hom}(E,E))\\
  \T \in I^{-\frac{5}{4}}(X \times \Sigma; C; \mathrm{Hom}(E,E)),
 \end{align*}
 where $C$ is the following canonical relation
 $$
  C = \{ (x,\xi, y,\eta) \in T^*X \times T^*\Sigma \mid (x,\xi) \sim (y,\eta) \}.
 $$
 Here $(x,\xi) \sim (y,\eta)$ if and only if there exists a lightlike vector $\tilde \eta$ in $T_y^* X$ such that
 $(x,\xi)$ and $(y,\tilde \eta)$ are in the same orbit of the geodesic flow on $X$ and
 the pull-back of $(y,\tilde \eta)$ to $\Sigma$ coincides with $(y,\eta)$.
\end{theorem}

The statement of the theorem can be implied directly from the calculus developed in \cite{hormander1971fourier} applied to matrix valued operators as
demonstrated by Dencker in \cite{dencker1982propagation}. A more direct proof for scalar valued operators  can be found in \cite[Ch.~5.1]{Duistermaat:1996aa}.
This proof carries over literally to the case of operators of real principal type and scalar valued principal symbol.
A simple computation shows that the characteristic Hamiltonian flow of half the principal symbol is the geodesic 
flow in lightlike direction. 

For each $(y,\eta) \in T^*\Sigma$ there is precisely one lightlike future/past directed co-vector $(y, \eta_\pm)$ in $T^* X$ whose pull-back to 
$\Sigma$ by the inclusion map is $(y, \eta)$. 
Thus, $C$ has two connected components $$C^\pm = \{ (x,\xi, y,\eta) \in C \mid \pm \xi \textrm{ is future directed} \}.$$
There is a canonical choice of trivialization of the Keller-Maslov $\mathbb{Z}_4$-principal bundle over $C$. Each connected component $C^\pm$ 
retracts to the set $C^\pm_\Sigma = \{ (x,\xi, x,\eta) \in C^\pm \}$, so it is enough to choose a trivialization on $C^\pm_\Sigma$. 
If an operator $U \in I^m(X \times \Sigma,C;\mathrm{Hom}(E,E))$ is given, then the restriction $\res_\Sigma \circ U$ is a pseudodifferential operator.
We define the trivialization of the Keller-Maslov line bundle such that its restriction to  $C^\pm_\Sigma$ coincides with the natural choice for pseudo-differential operators.
By construction each point in $C^\pm$ is in the orbit of a unique point in $C^\pm_\Sigma$ of the flow induced by the Hamiltonian vector field of half the principal symbol of $L$.
Therefore $C^\pm$ can be considered as a fibre space with each fibre isomorphic to some interval in $\R$ and the fibre coordinate induced by the flow
starting at the base $C^\pm_\Sigma$.
The base manifold $C^\pm_\Sigma$ is isomorphic to $T^*\Sigma$. We pull back the canonical volume form on $T^*\Sigma$ and use the Lebesgue
measure on the fibres to obtain a density $d_C$ on $C = C^+ \cup C^-$. The Lie derivative of the half-density $\sqrt{d_C}$ along the Hamiltonian vector field then vanishes.
We will use the half-density $\sqrt{2 \pi d_C}$ to trivialize the bundle $\Omega_{\frac{1}{2}}$ over $C$. This construction is essentially the same as the one used in \cite[Th. 6.6.1]{MR0388464}.

Assuming the Keller-Maslov bundle and the line bundle $\Omega_{\frac{1}{2}}$ have been trivialized
this way the principal symbols of Fourier integral operators in $I^m(X \times \Sigma,C;\mathrm{Hom}(E,E))$ can be identified with sections in $j_{C}^* \left( E \boxtimes F^* \right)$.
The subprincipal symbol of $L$ is then expressed in terms of the connection one form 
and thus the transport equation (5.1.11) in \cite{Duistermaat:1996aa} becomes the differential equation for parallel transport along lightlike geodesics.

In order to compute the principal symbols $\sigma_{\G}$ and $\sigma_{\T}$ of the Fourier integral operators $\G$ and $\T$ respectively, one first considers the system 
\begin{gather*}
 \res_\Sigma \circ \G = \id, \quad \res_\Sigma \circ (-\nabla_\nu ) \circ \G = 0, \\
 \res_\Sigma \circ \T = 0, \quad \res_\Sigma \circ (-\nabla_\nu ) \circ \T = \id
\end{gather*}
on the level of principal symbols. One obtains
$$\sigma_{\G}(y, \eta_\pm,y,\eta) = \frac{1}{2} \id, \quad \sigma_{\T}(y, \eta_\pm,y,\eta) =  \frac{\rmi}{2}  (\eta_\pm(\nu))^{-1} \id,$$
using $\eta_\pm(\nu)=\mp \|\eta\|$.
The principal symbols $\sigma_{\G}$ and $\sigma_{\T}$ of $\G$ and $\T$
can then be derived directly from the transport equations. If $(x,\xi)$ and $(y, \eta_\pm)$ are in the same orbit of the geodesic flow one thereby obtains
\begin{align}
\sigma_{\G}(x,\xi,y,\eta) = \frac{1}{2}\Gamma_{(x,\xi),(y,\eta_\pm)},\\
\sigma_{\T}(x,\xi,y,\eta) = \frac{\rmi}{2} (\eta_\pm(\nu))^{-1}\Gamma_{(x,\xi),(y,\eta_\pm)}, \label{princsym}
\end{align}
where $\Gamma_{(x,\xi),(y,\eta_\pm)}$ is the operator of parallel transport by the connection $\nabla$ along the lightlike geodesic connecting 
$(y,\eta_\pm)$ and $(x,\xi)$.

\bibliographystyle{plain}

\end{document}